\documentclass{amsart}

\usepackage{amssymb,latexsym,amsmath,extarrows,amsthm}
\usepackage{graphicx}
\usepackage{xcolor}
\usepackage{mathabx}

\newtheorem*{acknowledgement}{Acknowledgements}
\newtheorem{theorem}{Theorem}[section]
\newtheorem{lemma}[theorem]{Lemma}
\newtheorem{proposition}[theorem]{Proposition}
\newtheorem{remark}[theorem]{Remark}

\newtheorem{definition}[theorem]{Definition}
\newtheorem{corollary}[theorem]{Corollary}

\newtheorem{conjecture}[theorem]{Conjecture}

\newcommand{\al}{\alpha}
\newcommand{\be}{\beta}
\newcommand{\ga}{\gamma}

\newcommand{\de}{\delta}

\newcommand{\e}{\varepsilon}

\newcommand{\si}{\sigma}

\newcommand{\om}{\omega}

\newcommand{\cf}{\mathcal F}

\newcommand{\cy}{\mathcal Y}

\newcommand{\ZR}{\mathbb{R}}
\newcommand{\ZZ}{\mathbb{Z}}

\newcommand{\ZS}{\mathbb{S}}
\newcommand{\ZT}{\mathbb{T}}

\begin{document}

\title[Weighted restriction and Falconer's problem]{Weighted restriction estimates and application to Falconer distance set problem}

\author[X. Du]{Xiumin Du}
\address{
Institute for Advanced Study\\
Princeton, NJ}
\email{xdu@math.ias.edu}

\author[L. Guth]{Larry Guth}
\address{
Massachusetts Institute of Technology\\
Cambridge, MA}
\email{lguth@math.mit.edu}

\author[Y. Ou]{Yumeng Ou}
\address{
Massachusetts Institute of Technology\\
Cambridge, MA}
\email{yumengou@mit.edu} 

\author[H. Wang]{Hong Wang}
\address{
Massachusetts Institute of Technology\\
Cambridge, MA}
\email{hongwang@mit.edu} 

\author[B. Wilson]{Bobby Wilson}
\address{
Massachusetts Institute of Technology\\
Cambridge, MA}
\email{blwilson@mit.edu}

\author[R. Zhang]{Ruixiang Zhang}
\address{
Institute for Advanced Study\\
Princeton, NJ}
\email{rzhang@math.ias.edu}

\begin{abstract}
We prove some weighted Fourier restriction estimates using polynomial partitioning and refined Strichartz estimates. As application we obtain improved spherical average decay rates of the Fourier transform of fractal measures, and therefore improve the results for the Falconer distance set conjecture in three and higher dimensions.
\end{abstract}
    
\maketitle

\section{Introduction} \label{sec:Intro}
\setcounter{equation}0

In this article we prove improved partial result for Falconer distance set conjecture in dimension three and higher. Let $E\subset\mathbb{R}^d$ be a compact subset, its distance set $\Delta(E)$ is defined by 
$$
\Delta(E):=\{|x-y|:x,y\in E\}\,.
$$
In \cite{F}, Falconer conjectured that:

\begin{conjecture} [Falconer]
Let $d\geq 2$ and $E\subset\mathbb{R}^d$ be a compact set. Then 
$$
{\rm dim}(E)> \frac d 2 \Rightarrow |\Delta(E)|>0.
$$
Here $|\cdot|$ denotes the Lebesgue measure and ${\rm dim}(\cdot)$ is the Hausdorff dimension.
\end{conjecture}

Being open in every dimension, Falconer's conjecture has been studied by several authors, see Falconer \cite{F}, Mattila \cite{M}, Bourgain \cite{B}, Wolff \cite{W99} and Erdo\~gan \cite{Erdg04, Erdg06, Erdg05}. The previously best known results are that ${\dim}(E)> \frac d 2 +\frac 13$ implies $|\Delta(E)|>0$, due to Wolff \cite{W99} in dimension two and Erdo\~gan \cite{Erdg05} in dimension three and higher. Our main result is the following improvement:

\begin{theorem} \label{Falc}
Let $d\geq 3$ and $E\subset\mathbb{R}^d$ be a compact set with 
$$
{\rm dim}(E)> \alpha,\quad \alpha:=\begin{cases} 1.8, & d=3,\\ \frac d2 +\frac 14 +\frac{d+1}{4(2d+1)(d-1)}, & d\geq 4.\end{cases}
$$
Then $|\Delta(E)|>0$.
\end{theorem}

It is well known (see \cite{M, W99, Erdg05} for example) that Falconer's problem can be approached by weighted Fourier restriction (extension) estimates, which is the route we take in our proof. Consider the Fourier extension operator for the paraboloid 
$$
Ef(x):= \int_{B^{d-1}} e^{i(x'\cdot \om+x_d|\om|^2)} f(\om)\,d\om
$$
where $B^{d-1}$ denotes the unit ball in $\ZR^{d-1}$ and $x=(x',x_d)\in \ZR^d$.

Let $\cf_{\al,d}$ denote the collection of non-negative measurable functions $H:\ZR^d \rightarrow \ZR$ satisfying that 
\begin{equation}
\int_{B(x_0,r)} |H(x)| \,dx \leq r^\al, \quad \forall x_0\in\ZR^d,\quad \forall r\geq 1.
\end{equation}
\begin{remark}\label{unweighted}
Note that for $p\geq 1$, $H\in \cf_{\al,d}$, and functions $F$ with ${\rm supp } \widehat F \subset B^d$,
$$
\int |F|^p H\,dx \leq C_\al \int |F|^p \,dx\,.
$$We defer the justification of this observation to Subsection \ref{sec:justifyRmk}.
\end{remark}

We write $A\lessapprox B$ if $A\leq C_\e R^\e B$ for any $\e>0$, $R>1$ and let $B_R$ denote the ball of radius $R$. Theorem \ref{Falc} will be a consequence of the following weighted restriction estimates:

\begin{theorem} \label{wtRes-3}
Let $d=3$ and $\al \in (0,2]$. Then
\begin{equation} \label{eq:wtRes-3}
\|Ef\|_{L^3(B_R;Hdx)} \lessapprox \|f\|_{L^2}
\end{equation}
holds for all $f\in L^2(B^2)$, all $R> 1$ and all $H\in \cf_{\al,3}$.
\end{theorem}

\begin{theorem} \label{wtRes-d}
Let $d\geq 4$ and $\al \in [\frac d 2,\frac{d+1}{2}]$. 
Then
\begin{equation}
\|Ef\|_{L^{2d/(d-1)}(B_R;Hdx)} \lessapprox R^{\frac{\al}{2d^2}-\frac{1}{4d}}\|f\|_{L^2}
\end{equation}
holds for all $f\in L^{2}(B^{d-1})$, all $R> 1$ and all $H\in\cf_{\al,d}$.
\end{theorem}

Weighted restriction estimates in the vein of the above can be used to prove Falconer's problem via a famous scheme due to Mattila. Briefly speaking, Theorem \ref{wtRes-3}, \ref{wtRes-d} imply improved estimates for spherical average decay rates of the Fourier transform of fractal measures, from which improved results for the Falconer distance set conjecture follow. We leave the detailed discussion of Mattila's approach to Section \ref{sec:Pre}.

In fact, we obtain not only the aforementioned Fourier decay estimates for fractal measures immediately implied by Theorem \ref{wtRes-3}, \ref{wtRes-d}, but also ones corresponding to $\alpha$ in the whole range $(0,d]$. More precisely,
\begin{definition}
A compactly supported probability measure $\mu$ is called $\alpha$-dimensional if it satisfies
\begin{equation}
\mu(B(x,r))\leq C_\mu r^\al, \quad \forall r>0, \quad \forall x\in \ZR^d.
\end{equation}
\end{definition}

Let $\be_d(\al)$ denote the supremum of the numbers $\be$ for which 
\begin{equation} \label{eq:AvrDec}
\left\|\widehat \mu (R\cdot)\right\|_{L^2(\ZS^{d-1})}^2 \leq C_{\al,\mu} R^{-\be}
\end{equation}
whenever $R> 1$ and $\mu$ is $\al$-dimensional.

The problem of identifying the precise value of $\beta_d(\al)$ was proposed by Mattila \cite{M04}. 
In two dimensions, the sharp decay rates are known:
$$
\be_2(\al)=
\begin{cases}
\al, & \al\in(0,\,1/2], \quad \text{(Mattila \cite{M})}\\
1/2, & \al\in[1/2,\,1], \quad \text{(Mattila \cite{M})}\\
\al/2, & \al\in[1,\,2], \quad \text{(Wolff \cite{W99})}.
\end{cases}
$$
The problem remains open when $d\geq 3$ and $\al>\frac{d-1}{2}$.
See Luc\`a-Rogers \cite{LR} and the references therein for example for a discussion of various partial results.
In higher dimensions, the previously best known lower bounds are 
$$
\be_d(\al)\geq
\begin{cases}
\al, & \al\in(0,\,\frac{d-1}{2}], \quad \text{(Mattila \cite{M})}\\
\frac{d-1}{2}, & \al\in[\frac{d-1}{2},\,\frac d 2], \quad \text{(Mattila \cite{M})}\\
\al-1+\frac{d+2-2\al}{4}, & \al\in[\frac d2, \,\frac d2+\frac 23+\frac 1d], \quad \text{(Erdo\~gan \cite{Erdg05})} \\
\al-1+\frac{(d-\al)^2}{(d-1)(2d-\al-1)}, & \al\in[\frac d2+\frac 23+\frac 1d, \,d], \quad \text{(Luc\`a-Rogers \cite{LR})}.
\end{cases}
$$

For $d\geq 3$, we obtain the following lower bound of $\beta_d(\alpha)$, $\alpha\in(0,d]$, which improves the previously best known results above for all $\al\in (d/2,d)$: 
\begin{theorem}\label{AvrDec}
Let $d\geq 3$ and $\alpha\in (0,d]$. Then, if $d=3$,
$$
\be_3(\al)\geq \begin{cases}\frac{2\al}{3}, & \alpha\in (0,2],\\ \frac{4}{3}, & \alpha\in (2,\frac{19}{9}],\\ \frac{3}{4}\alpha-\frac{1}{4}, &\alpha\in (\frac{19}{9},3], \end{cases}
$$if $d\geq 4$,
$$
\beta_d(\alpha)\geq \max\left(\beta_d^0(\alpha), \alpha-1+\frac{d-\alpha}{d+1}\right), 
$$with $\beta_d^0(\alpha)$ defined as
$$
\beta_d^0(\alpha):=\begin{cases} \frac{(d-1)^2}{d}, & \alpha\in(\#_d, d],\\ 
\frac{(2d-3-2S_4^d)\alpha}{2d}+\frac{1}{d}+S_4^d, & \alpha\in (d-1,\#_d],\\ 
\frac{(d-1-S_4^d)\alpha}{d}-\frac{1}{2}+\frac{3}{2d}+S_4^d,& \alpha\in (d-2, d-1],\\
\frac{(d-1-S_\ell^d)\alpha}{d}-\frac{1}{2}+\frac{\ell-1}{2d}+S_\ell^d, & \alpha\in (d-\frac{\ell}{2}, d-\frac{\ell}{2}+\frac{1}{2}],\,\forall 5\leq\ell\leq d,\\
\frac{(d-1)\alpha}{d}, & \alpha\in (0,\frac{d}{2}],
\end{cases}
$$where $S_\ell^d:=\sum_{i=\ell}^d\frac{1}{i}$ and $\#_d:=\frac{2d(d-2-S_4^d)}{2d-3-2S_4^d}$.
\end{theorem}

In the higher dimensional case of Theorem \ref{AvrDec} above, $\alpha-1+\frac{d-\alpha}{d+1}$ gives the better lower bound if $\alpha$ is large while $\beta^0_d(\alpha)$ is better if $\alpha$ is small. We also point out that for $\alpha\in (0,d/2)$, the lower bound in Theorem \ref{AvrDec} is not as good as the result of Mattila \cite{M}.

One of the main steps in the proof of Theorem \ref{AvrDec} is demonstrating the following weighted restriction estimates, which are natural extensions of Theorem \ref{wtRes-3}, \ref{wtRes-d} to all $\alpha\in (0,d]$.

\begin{theorem}\label{AvrDec-wtRes}
Let $d\geq 3$ and $\alpha\in (0,d]$. Then, 
\[
\|Ef\|_{L^{2d/(d-1)}(B_R;Hdx)}\lessapprox R^{\gamma^0_d(\alpha)}\|f\|_{L^2}
\]holds for all $f\in L^2(B^{d-1})$, all $R>1$ and all $H\in\mathcal{F}_{\alpha,d}$, where
\[
\gamma^0_{3}(\alpha):=\begin{cases}
0, & \alpha\in (0,2],\\
\frac{\alpha}{3}-\frac{2}{3}, & \alpha\in (2,3],
\end{cases}
\]
\[
\gamma^0_d(\alpha):=\begin{cases}
\frac{(d-1)(\alpha+1-d)}{2d},& \alpha\in (\#_d,  d],\\
\frac{(1+2S_4^d)\alpha}{4d}-\frac{1}{2d}-\frac{S^d_4}{2},& \alpha\in (d-1, \#_d],\\
\frac{S^d_4 \alpha}{2d}+\frac{1}{4}-\frac{3}{4d}-\frac{S^d_4}{2},& \alpha\in (d-2, d-1],\\
\frac{S^d_\ell \alpha}{2d}+\frac{1}{4}-\frac{\ell-1}{4d}-\frac{S^d_\ell}{2}, &\alpha\in (d-\frac{\ell}{2}, d-\frac{\ell}{2}+\frac{1}{2}],\,\forall 5\leq\ell\leq d,\\
0,& \alpha\in (0,\frac{d}{2}],
\end{cases}\quad (d\geq 4)
\]and $S_\ell^d$, $\#_d$ are defined as in Theorem \ref{AvrDec} above.
\end{theorem}

\begin{remark}
Note that by Tomas-Stein restriction theorem and H\"older's inequality, for $d\geq 3$ we have 
$$
\|Ef\|_{L^{2d/(d-1)}(B_R;Hdx)}\lesssim R^{\frac{\al(d-1)}{2d(d+1)}}\|f\|_{L^2}
$$
for all $f\in L^2(B^{d-1})$, all $R>1$ and all $H\in\mathcal{F}_{\alpha,d}$. This estimate is better than Theorem \ref{AvrDec-wtRes} when $\al \geq d-\frac{1}{d}$. In our approach, when $\al$ is large, the exponent $\ga_d^0(\al) = \frac{(d-1)(\al+1-d)}{2d}$ comes from the constraint of parabolic rescaling when reducing the linear estimate to a (weak) bilinear one, and that estimate is not good enough. While for application to the average decay rates, we prove a linear $L^2$ estimate using refined Strichartz directly, which gives the decay rates $\al-1+\frac{d-\al}{d+1}$ in Theorem \ref{AvrDec} and improves previous best known results when $\al$ is large (see Section \ref{RS} for details).
\end{remark}

\begin{remark} \label{sphere}
It follows from the variable-coefficient generalization as discussed in \cite{lG16}, that the same weighted restriction estimates in Theorem \ref{wtRes-3}, \ref{wtRes-d}, \ref{AvrDec-wtRes} above still hold true if one replaces the paraboloid by sphere or other positively curved hypersurfaces. In particular, to deduce Theorem \ref{Falc} and \ref{AvrDec} from the newly obtained weighted restriction estimates using Mattila's approach, as described in Subsection \ref{sec:Mattila} below, it is fine to replace the paraboloid in the weighted restriction estimates by the sphere. 
\end{remark}

The estimates of the Fourier decay rate of fractal measures in Theorem \ref{AvrDec} also imply the following improved result for the pinned distance set problem, by applying Theorem 1.4 of a very recent work of Liu \cite{Liu}.
\begin{corollary}
Let $d\geq 3$ and $E\subset \mathbb{R}^d$ be a compact set with
$$
{\rm dim}(E)> \alpha,\quad \alpha:=\begin{cases} 1.8, & d=3,\\ \frac d2 +\frac 14 +\frac{d+1}{4(2d+1)(d-1)}, & d\geq 4.\end{cases}
$$
Then there exists $x\in E$ such that its pinned distance set
\[
\Delta_x(E):=\{|x-y|:\,y\in E\}
\]has positive Lebesgue measure.
\end{corollary}

In addition, Theorem \ref{AvrDec} implies directly improved upper bounds of the Hausdorff dimension of divergence sets of solutions to wave equations, by applying \cite[Proposition 1.5]{LR}. We omit the details.

The key ingredients in our proofs are the method of polynomial partitioning developed by the second author \cite{lG} \cite{lG16}
and (linear and bilinear) refined Strichartz estimates obtained by Li and the first two authors in \cite{DGL}. Polynomial partitioning has proved to be extremely powerful in the study of restriction type problems such as the restriction estimates for the paraboloid \cite{lG, lG16} , the cone \cite{OW} and H\"ormander-type oscillatory integral operators \cite{GHI}. The sharp Schr\"odinger maximal estimate in $\mathbb{R}^2$ \cite{DGL} was also recently derived via the polynomial partitioning scheme, combined with the aforementioned refined Strichartz estimates.

Compared to \cite{Erdg05}, where the previously best known result for Falconer's problem in $d\geq 3$ was proved via a similar route through weighted restriction estimates, our argument has the following advantages. First, the use of polynomial partitioning enables one to obtain a more delicate estimate by inducting on dimensions and extracting information from every intermediate dimension. Second, in every fixed intermediate dimension, compared to H\"older's inequality that is used in \cite{Erdg05}, the (linear and bilinear) refined Strichartz estimates provide much finer estimates. The latter advantage is particularly important for deriving the three-dimensional case (Theorem \ref{wtRes-3}), where there is not much information available from lower dimensions while the bilinear refined Strichartz estimate plays a key role. 

The structure of the paper is as follows. In Section \ref{sec:Pre}, we review preliminaries including parabolic rescaling, wave packet decomposition and Matilla's approach, and explain the connections between Theorem \ref{wtRes-3}, \ref{wtRes-d}, \ref{AvrDec-wtRes}, \ref{AvrDec} and how they imply Theorem \ref{Falc}. In Section \ref{RS}, we review linear and bilinear refined Strichartz estimates, and obtain some partial improvements towards Falconer's distance set problem and average decay rates. 
In Section \ref{sec:wtRes-3}, we prove Theorem \ref{AvrDec-wtRes} in the case $d=3$, using polynomial partitioning and bilinear refined Strichartz. The proof of Theorem \ref{wtRes-d} is presented in Section \ref{sec:wtRes-d}, and additional ingredients that are needed in generalizing it to Theorem \ref{AvrDec-wtRes} when $d\geq 4$ are discussed in Section \ref{sec:wtRes-d-al}. \\

\noindent \textbf{List of Notations:}

We write $A\lessapprox B$ if $A\leq C_\e R^\e B$ for any $\e>0$, $R>1$; $A\lesssim_\e B$ if $A \leq C_\e B$, $A \lesssim_{K,\e} B$ if $A \leq C_{K,\e} B$, etc; $A\lesssim B$ if $A \leq C B$ for a constant $C$ which only depends on some unimportant fixed variables such as $d,\al$ and sometimes $\e$ too. 

For each $\e>0$, there is a sequence of small parameters
$$
\delta_{\rm deg} \ll \delta\ll\delta_{d-1}\ll\delta_{d-2}\ll\cdots\ll\delta_1\ll\delta_0\ll\e.
$$
For $Z=Z(P_1,\cdots,P_{d-m})$, $D_Z$ denotes an upper bound of the degrees of $P_1,\cdots,P_{d-m}$. Usually $D_Z\leq R^{\delta_{\rm deg}}$ unless noted otherwise.

Let $m$ be a dimension in the range $1\leq m\leq d$. Denote $$r_m:=\frac{2(m+1)}{m}=p_{m+1} < p_m:=\frac{2m}{m-1} < q_m:=\frac{2(m+1)}{m-1}\,.$$
Let $B^m_R$ stand for a ball of radius $R$ in $\ZR^m$, $B^m$ denote the unit ball in $\ZR^m$ and $B_R$ abbreviate $B^d_R$ for simplicity.

\begin{acknowledgement}
The work of X. Du is supported by the National Science Foundation under Grant No. 1638352 and the Shiing-Shen Chern Fund. L. Guth is supported by a Simons Investigator Award. The work of R. Zhang is supported by the National Science Foundation under Grant No. 1638352 and the James D. Wolfensohn Fund.
\end{acknowledgement}

\section{Preliminaries} \label{sec:Pre}
\setcounter{equation}0

\subsection{Parabolic rescaling}

\begin{lemma} \label{ParResc}
There exists an absolute constant $C$ so that the following holds true.
Let $p\geq 1$,$\al \in (0,d]$ and $\tilde R$ be a sufficiently large constant. Suppose that 
\begin{equation} \label{resc0}
\|Ef\|_{L^p(B_{R};Hdx)} \leq \tilde C  R^{\ga} \|f\|_{L^2}
\end{equation}
holds for all $f\in L^2(B^{d-1})$, all $1<R\leq \tilde R/2$ and all $H\in\cf_{\al,d}$.
Then
\begin{equation*}
\|Ef\|_{L^p(B_R;Hdx)} \leq C \tilde C  K^{\frac{\al+1}{p}-\frac{d-1}{2}-\ga}R^{\ga} \|f\|_2
\end{equation*}
holds for all $H\in\cf_{\al,d}$, all $1<R\leq \tilde R$ and all $f \in L^2$ with support in some ball of radius $1/K$ inside $B^{d-1}$, where $K$ is any large constant $<\tilde R$.
\end{lemma}

\begin{proof}
Let $H\in \cf_{\al,d}$ and $f\in L^2$ with ${\rm supp} f \subset B(\om_0,1/K)\subset B^{d-1}$. We write $\om = \om_0 + 
\frac 1 K\xi \in B(\om_0, 1/K)$, then by change of variables,
$$
|Ef(x',x_d)|= \frac{1}{K^{(d-1)/2}}|Eg(y',y_d)|\,,
$$
where $g\in L^2(B^{d-1})$ with $\|g\|_2=\|f\|_2$, more precisely,
$$
g(\xi)=\frac{1}{K^{(d-1)/2}}f(\om_0+\frac1K \xi) \,,
$$
and the new coordinates $(y',y_d)$ are related to the old coordinates $(x',x_d)$ by 
\begin{equation*}
\begin{cases}
y'=\frac 1K x' +\frac{2x_d}{K}\om_0\,,\\
y_d=\frac{x_d}{K^2}\,.
\end{cases}
\end{equation*}
For simplicity, we denote the relation above by $y=T(x)$.
Therefore,
$$
\|Ef(x)\|_{L^p(B_R;H(x)dx)}=K^{\frac{d+1}{p}-\frac{d-1}{2}+\frac{\al-d}{p}} \|Eg(y)\|_{L^p(\tilde B; H^*(y)dy)}\,, 
$$
where $\tilde B=T(B_R)$ is contained in a box of dimensions $\sim \frac RK \times \cdots \times \frac RK \times \frac{R}{K^2}$, and the function $H^*$ is given by
$$
H^*(y)=K^{d-\al}H(T^{-1}y)\,.
$$
Note that for $\forall x_0\in \ZR^d, \forall r\geq 1$, 
$$
\int_{B(x_0,r)} H^*(y) \,dy = K^{d-\al}K^{-(d+1)} \int_{\bar B} H(x) \, dx\,,
$$
where $\bar B = T^{-1} (B(x_0,r))$ is contained in $K$ balls of radius $\sim Kr$, hence it follows from $H\in \cf_{\al,d}$ that
$$
\int_{B(x_0,r)} H^*(y) \,dy  \lesssim K^{-\al-1} K (Kr)^{\al}=r^\al\,,
$$
i.e. $H^*\in \cf_{\al,d}$ up to a constant. Applying \eqref{resc0} to functions $g$ and $H^*$ and physical radius $R/K$ we obtain
$$
\|Ef\|_{L^p(B_R;Hdx)} \lesssim
\tilde C K^{\frac{\al+1}{p}-\frac{d-1}{2}-\ga}R^{\ga} \|f\|_{L^2}\,.
$$
This completes the proof.
\end{proof}

\subsection{Mattila's approach}\label{sec:Mattila}

Our study of Falconer's distance set problem follows a scheme that goes back to Mattila \cite{M}. We briefly recall this approach here. See also for example Lemma 2.1 in \cite{Erdg06}.

Let $d\sigma$ be the $(d-1)$-dimensional surface measure on $S^{d-1}$ and $E_{S^{d-1}}$ stand for the extension operator over the unit sphere $S^{d-1}$.

\begin{theorem}[Mattila \cite{M}]\label{Mthm}
Fix $\alpha\in (d/2,d)$. Assume that for all $\alpha$-dimensional compactly supported probability measure $\mu$ there holds
\begin{equation}\label{Mcondition}
\|\widehat{\mu}(R\cdot)\|_{L^2(S^{d-1})}\leq C_\mu R^{\frac{\alpha-d}{2}},\quad \forall R>1.
\end{equation}
Then Falconer's conjecture holds for $\alpha$, i.e. for any compact subset $E$ of $\ZR^d$,
\[
{\rm dim}(E)>\al \Rightarrow |\Delta(E)|>0.
\]
\end{theorem}

\begin{proof}[Sketch of the proof of Theorem \ref{Mthm}]
If $E$ is a compact subset of $\ZR^d$ with $\dim E > \al$, then by Frostman's lemma $E$ supports an $(\al+\e_0)$-dimensional measure $\mu$, for some $\e_0>0$. In particular, $\mu$ is also $\al$-dimensional and the $\al$-dimensional energy of $\mu$ is finite:
$$
I_\al(\mu) := \int \int |x-y|^{-\al} d\mu(x)d\mu(y)<\infty\,.
$$
We have by assumption \eqref{Mcondition}
\begin{equation}
\begin{split}
&\int_{1}^{\infty} \left(\int_{S^{d-1}} |\widehat{\mu}(Rx)|^2 d\sigma(x)\right)^2R^{d-1}dR \\
\lesssim &\int_{1}^{\infty} \left(\int_{S^{d-1}} |\widehat{\mu}(Rx)|^2 d\sigma(x)\right) R^{\al - d}\cdot R^{d-1}dR 
\sim I_{\alpha} (\mu) < \infty\,,
\end{split}
\end{equation}
where the last equivalence follows from the Fourier representation of the energy.

Matilla proved that this estimate is equivalent to some measure supported on $\Delta(E) \bigcup -\Delta (E)$ having its Fourier transform in $L^2 (\ZR)$. This in turn implies $|\Delta(E)|>0$. See also Section B in Chapter 9 of Wolff \cite{WHA}.
\end{proof}

\begin{proposition}\label{wtRestoFalc}
Let $\alpha\in (0,d)$, $p\geq 1$, and $E_{S^{d-1}}g:=(gd\sigma)^\vee$. Suppose that
\begin{equation}\label{eqn:wtRes}
\|E_{S^{d-1}} g\|_{L^p(B_R;Hdx)} \lesssim  R^{\ga+\epsilon} \|g\|_{L^2},\quad \forall R>1
\end{equation}holds for all $g\in L^2(S^{d-1})$ and all $H\in\cf_{\al,d}$, and that
\begin{equation}\label{cri:alpha}
\gamma \leq \alpha (\frac{1}{p}+\frac{1}{2})-\frac{d}{2}.
\end{equation}Then, Falconer's conjecture holds for $\alpha$, i.e.
\[
{\rm dim}(E)>\al \Rightarrow |\Delta(E)|>0.
\]
\end{proposition}
\begin{proof}

The proof is essentially contained in Wolff \cite{W99} and Erdo\~gan \cite{Erdg06}. We follow their treatment here.

Given (\ref{eqn:wtRes}), it suffices to verify the averaged decay estimate (\ref{Mcondition}) and apply Theorem \ref{Mthm}. Without loss of generality we assume $\mu$ is supported in the unit ball. We use a duality argument. Take an arbitrary $f$ supported on the unit sphere $S^{d-1}$. By (\ref{eqn:wtRes}), we have for all $H\in\cf_{\al,d}$,
\begin{equation}\label{Lpnormextineq}
\left(\int_{B_R} |(f d\sigma)^\vee|^p H dx\right)^{\frac{1}{p}} \lesssim R^{\gamma+\epsilon}\|f\|_{L^2(S^{d-1}; d\sigma)}.
\end{equation}

Now take a radial Schwartz bump function $\psi$ such that $\psi(x)=1$ for all $|x|=1$, and that $\widehat{\psi}$ has compact support. Notice that $R^{\al} \cdot \mu(\frac{\cdot}{R}) \ast |\widehat{\psi}|$ is a function in $\mathcal{F}_{\al, d}$, where the dilated measure $\mu(\frac{\cdot}{R})$ is defined as
\[
\int g(x)\,d\mu(\frac{\cdot}{R}):=\int g(Rx)\,d\mu.
\]Indeed, for any $r\geq 1$, the measure of any ball of radius $r$ with respect to $R^{\al} \cdot \mu(\frac{\cdot}{R})\ast |\widehat{\psi}|$ is $\leq R^{\al} (\frac{r}{R})^{\al} = r^{\al}$. Moreover, $R^{\al} \cdot \mu(\frac{\cdot}{R}) \ast |\widehat{\psi}|$ has its $L^1$ norm $\lesssim R^{\al}$ times the total measure of $\mu$, which is in turn bounded by a constant times $R^{\al}$.

Apply (\ref{Lpnormextineq}) to the function $H = R^{\al} \cdot \mu(\frac{\cdot}{R}) \ast |\hat{\psi}|$ and use H\"older's inequality, we obtain
\begin{equation}
\int |(f d\sigma)^\vee| H dx \lesssim R^{\gamma+\frac{\al}{p'}+\epsilon}\|f\|_{L^2(S^{d-1}; d\sigma)}
\end{equation}
or
\begin{equation}
\int |(f d\sigma)^\vee| (R^{\al} \cdot \mu(\frac{\cdot}{R}) \ast |\widehat{\psi}|) dx \lesssim R^{\gamma+\frac{\al}{p'}+\epsilon}\|f\|_{L^2(S^{d-1}; d\sigma)}.
\end{equation}

Since $\psi=1$ on the unit sphere, $(f d\sigma)^\vee * \widehat{\psi} = (f d\sigma)^\vee$. Hence,
\begin{equation}\label{L1normextineq}
\int |(f d\sigma)^\vee| d \mu(\frac{\cdot}{R}) \lesssim R^{\gamma-\frac{\al}{p}+\epsilon}\|f\|_{L^2(S^{d-1}; d\sigma)}.
\end{equation}

Note that the measure $d \mu(\frac{\cdot}{R})$ has Fourier transform $\widehat{\mu} (R\cdot)$. By duality and Matilla's Theorem \ref{Mthm}, Falconer's conjecture holds for $\alpha$ as long as $\gamma-\frac{\al}{p} \leq \frac{\al -d}{2}$ (we removed the $\epsilon$ here because when $\dim(E) > \alpha$, it is also $>$ some $\al + \epsilon$). This is equivalent to $\gamma \leq \alpha (\frac{1}{p}+\frac{1}{2})-\frac{d}{2}$, as claimed in (\ref{cri:alpha}).
\end{proof}

It is now clear, due to Proposition \ref{wtRestoFalc}, that Theorem \ref{Falc} follows directly from Theorem \ref{wtRes-3} and \ref{wtRes-d}.

\begin{remark}
As noted in Wolff \cite{W99} and Erdo\~gan \cite{Erdg05}, Mattila's approach described above cannot be used to prove the full Falconer's conjecture in dimension 2 or 3. The best possible exponents it would imply are $\frac{4}{3}$ and $\frac{5}{3}$, respectively. However, one might be able to prove Falconer's conjecture in dimension $d \geq 4$ using this method.
\end{remark}

\begin{remark}\label{rmk:wtRestoAD}
From the proof of the proposition above, one concludes directly that under the assumption of Proposition \ref{wtRestoFalc} except for $(\ref{cri:alpha})$, there holds the lower bound estimate for the Fourier decay rates of fractal measures
\begin{equation}
\beta_d(\alpha)\geq 2(\frac{\alpha}{p}-\gamma)\,,
\end{equation}
where $\beta_d(\alpha)$ is as defined in (\ref{eq:AvrDec}). From this we see that Theorem \ref{AvrDec} follows from Theorem \ref{AvrDec-wtRes} and (\ref{AvrDec0}) below.
\end{remark}

\subsection{Proof of Remark \ref{unweighted}}\label{sec:justifyRmk}
For the sake of completeness, we give a justification of Remark \ref{unweighted} in this subsection.

Let $\psi$ be a Schwartz bump function such that $\psi=1$ on the unit ball $B^d$, hence $F=F\ast \widecheck{\psi}$ as ${\rm supp }\widehat{F}\subset B^d$. Therefore for all $p\geq 1$, by H\"older's inequality,
\[
\begin{split}
\int |F|^pH\,dx &=\int \left|\int F(y)\widecheck{\psi}(x-y)\,dy \right|^p H(x)\,dx\\
&\leq \int \left(\int |F(y)|^p |\widecheck{\psi}(x-y)|\,dy\right)\left(\int|\widecheck{\psi}(x-y)|\,dy\right)^{p-1} H(x)\,dx\\
&\lesssim \int |F(y)|^p \left(\int |\widecheck{\psi}(x-y)|H(x)\,dx\right)\,dy.
\end{split}
\]Observe that for any $y\in\mathbb{R}^d$ and sufficiently large $M=M(\alpha)>0$,
\[
\begin{split}
\int |\widecheck{\psi}(x-y)|H(x)\,dx&\leq C_M\sum_{j=0}^\infty \int \chi_{B(y,2^j)}(x)2^{-jM}H(x)\,dx\\
&\lesssim C_M\sum_{j=0}^\infty 2^{j(\alpha-M)}<\infty,
\end{split}
\]where we have used the fact that $H\in\mathcal{F}_{\alpha,d}$. Hence the desired estimate follows.

\subsection{Wave packet decomposition} \label{wpd}

We use the same setup as in Section 3 of \cite{lG16}, which we briefly recall here. Let $f$ be a function on $B^{d-1}$, we break it up into pieces $f_{\theta,\nu}$ that are essentially localized in both position and frequency. Cover $B^{d-1}$ by finitely overlapping balls $\theta$ of radius $R^{-1/2}$ and cover $\ZR^{d-1}$ by finitely overlapping balls of radius $R^{\frac{1+\delta}{2}}$, centered at $\nu \in R^{\frac{1+\delta}{2}}\ZZ^{d-1}$. Using partition of unity, we have a decomposition
$$
f=\sum_{(\theta,\nu)\in \ZT} f_{\theta,\nu} + {\rm RapDec}(R)\|f\|_{L^2}\,,
$$
where $f_{\theta,\nu}$ is supported in $\theta$ and has Fourier transform essentially supported in a ball of radius $R^{1/2+\delta}$ around $\nu$. The functions $f_{\theta,\nu}$ are approximately orthogonal. In other words, for any set $\ZT'\subset \ZT$ of pairs $(\theta,\nu)$, we have
$$
\big\|\sum_{(\theta,\nu)\in \ZT'} f_{\theta,\nu}\big\|_{L^2}^2
\sim \sum_{(\theta,\nu)\in \ZT'} \|f_{\theta,\nu}\|_{L^2}^2\,.
$$
For each pair $(\theta,\nu)$, the restriction of $Ef_{\theta,\nu}$ to $B_R$ is essentially supported on a tube $T_{\theta,\nu}$ with radius $R^{1/2+\delta}$ and length $R$, with direction $G(\theta)\in S^{d-1}$ determined by $\theta$ and location determined by $\nu$, more precisely,
$$
T_{\theta,\nu} :=\left\{(x',x_d) \in B_R : |x'+2x_d\omega_\theta -\nu|\leq R^{1/2+\delta}\right\}\,.
$$
Here $\omega_\theta \in B^{d-1}$ is the center of $\theta$, and 
$$
G(\theta)=\frac{(-2\omega_\theta,1)}{|(-2\omega_\theta,1)|}\,.
$$

In our proof, a key concept is a wave packet being \emph{tangent} to an algebraic variety. 
We write $Z(P_1,\cdots,P_{d-m})$ for the set of common zeros of the polynomials $P_1,\cdots,P_{d-m}$. The variety $Z(P_1,\cdots,P_{d-m})$ is called a \emph{transverse complete intersection} if 
$$
\nabla P_1(x) \wedge \cdots \wedge \nabla P_{d-m}(x) \neq 0 \text{ for all } x\in Z(P_1,\cdots,P_{d-m})\,.
$$
Let $Z$ be an algebraic variety and $E$ be a positive number. For any tile $(\theta,\nu) \in\ZT$,
we say that $T_{\theta,\nu}$ is \emph{$E R^{-1/2}$-tangent} to $Z$ if
$$T_{\theta,\nu}\subset N_{E R^{1/2}}Z \cap B_R,\quad and$$
\begin{equation*}
 \text{Angle}(G(\theta),T_zZ)\leq E R^{-1/2}
\end{equation*}
for any non-singular point $z\in N_{2 E R^{1/2}} ( T_{\theta,\nu}) \cap 2B_R \cap Z$.

Let
$$
\ZT_Z (E):=\{(\theta,\nu)\in\ZT\,|\,T_{\theta,\nu} \text{ is $E R^{-1/2}$-tangent to}\, Z\}\,,
$$
and  we say that $f$ is concentrated in wave packets from $\ZT_Z(E)$ if
$$
 \sum_{(\theta,\nu)\notin \ZT_Z(E)} \|f_{\theta,\nu}\|_{L^2} \leq {\rm RapDec}(R)\|f\|_{L^2}.
$$
Since the radius of $T_{\theta, \nu}$ is $R^{1/2 + \delta}$, $R^\delta$ is the smallest interesting value of $E$.

\section{Linear and bilinear refined Strichartz estimates\\ in higher dimensions}\label{RS}
\setcounter{equation}0

One of the key ingredients in our proof is the linear and bilinear refined Strichartz estimates established in \cite{DGL}. 

\begin{theorem} [Linear refined Strichartz for $m$-variety in $d$ dimensions] \label{thm-linref}
Let $d\geq 2$ and $m$ be a dimension in the range $2 \leq m \leq d$. Let $q_m=2(m+1)/(m-1)$. Suppose that $Z=Z(P_1,\cdots,P_{d-m})$ is a transverse complete intersection where ${\rm Deg}\,P_i \leq D_Z$. Suppose that $f\in L^2(B^{d-1})$ is concentrated in wave packets from $\ZT_Z (E).$ 
Suppose that $Q_1, Q_2, ...$
are lattice  $R^{1/2}$-cubes in $B_R$, so that
$$ \| Ef \|_{L^{q_m}(Q_j)} \textrm{ is essentially constant in $j$}. $$
\noindent Suppose that these cubes are arranged in horizontal strips of the form $\ZR \times \cdots \times \ZR \times \{t_0, t_0 + R^{1/2} \}$, and that each such strip contains $\sim \sigma$ cubes $Q_j$.  Let $Y$ denote $\bigcup_j Q_j$. Then
\begin{equation}
\label{linref}\| Ef \|_{L^{q_m}(Y)} \lessapprox  E^{O(1)}\sigma^{-\frac{1}{m+1}}R^{-\frac{d-m}{2(m+1)}}  \| f \|_{L^2}.
\end{equation}
\end{theorem}

\begin{theorem} [Bilinear refined Strichartz for $m$-variety in $d$ dimensions] \label{thm-bilref-m} 
Let $d\geq 2$ and $m$ be a dimension in the range $2 \leq m \leq d$. Let $q_m=2(m+1)/(m-1)$. For functions $f_1$ and $f_2$ in $L^2(B^{d-1})$, with supports separated by $\sim 1$, suppose that $f_1$ and $f_2$ are concentrated in wave packets from $\ZT_Z(E)$, where $Z=Z(P_1,\cdots,P_{d-m})$  is a transverse complete intersection with ${\rm Deg}\,P_i \leq D_Z$.
Suppose that $Q_1, Q_2, \cdots, Q_N$ are lattice  $R^{1/2}$-cubes in $B_R$, so that for each $i$,
$$ \| E f_i \|_{L^{q_m}(Q_j)} \textrm{ is essentially constant in $j$}. $$
Let $Y$ denote $\bigcup_{j=1}^{N} Q_j$. Then
\begin{equation} \label{bilref-m}
\left\| |E f_1|^{1/2}| Ef_2 |^{1/2} \right\|_{L^{q_m}(Y)} \lessapprox E^{O(1)} R^{-\frac{d-m}{2(m+1)}} N^{-\frac{1}{2(m+1)}} \| f_1 \|_{L^2}^{1/2}\| f_2 \|_{L^2}^{1/2}.
\end{equation}
\end{theorem}

Theorem \ref{thm-linref} and Theorem \ref{thm-bilref-m}
were proved in \cite{DGL} in the case $m=2,d=3$, via the Bourgain-Demeter $l^2$-decoupling theorem \cite{BD} and induction on scales. The proof for general $m$ and $d$ follows from exactly the same lines, with only changes in numerology, thus we skip the proof and refer interested readers to Section 7 in \cite{DGL}.

The following weighted linear and bilinear restriction estimates are immediate consequences of Theorem \ref{thm-linref} and Theorem \ref{thm-bilref-m}.

\begin{corollary} [Linear weighted $L^2$ estimate] \label{cor:linL2}
Let $d\geq 2$ and $\al\in (0,d]$. Let $m$ be a dimension in the range $2\leq m \leq d$. Suppose that $Z=Z(P_1,\cdots,P_{d-m})$ is a transverse complete intersection where ${\rm Deg}\,P_i \leq D_Z$, and that $f\in L^2(B^{d-1})$ is concentrated in wave packets from $\ZT_Z(E)$ and $H\in \cf_{\al,d}$. 
Then
\begin{equation} \label{linL2}
\|Ef\|_{L^2(B_R;Hdx)} \lessapprox E^{O(1)} R^{\frac 1 2 -\frac{d-\al}{2(m+1)}}\|f\|_{L^2}\,.
\end{equation}
\end{corollary}

\begin{proof} 
Without loss of generality, assume that $\|f\|_{L^2} = 1$. We break $B_R$ into $R^{1/2}$-cubes $Q_j$. 
Let $\cy_{\ga,\si}$ denote the collection of those $Q_j$'s such that 
\begin{align*}
\bullet \,&\|Ef\|_{L^{q_m}(Q_j)} \sim \ga, \\
\bullet \,&\text{the horizontal $R^{1/2}$-strip containing $Q_j$ contains $\sim \si$ $R^{1/2}$-cubes satisfying} \\ 
&\text{the above condition}.
\end{align*}
Define $Y _{\ga,\si}:= \bigcup_{Q_j \in \cy_{\ga,\si}} Q_j$. Note that there are only $\sim (\log R)^2$ relevant dyadic scales $(\ga,\si)$, hence
$$
\|Ef\|_{L^2(B_R;Hdx)}
\lesssim (\log R)^2
\|Ef\|_{L^2(Y;Hdx)}\,,
$$
where $Y=Y_{\ga,\si}$ for some $(\ga,\si)$. Therefore by H\"older's inequality and Theorem \ref{thm-linref} one has
$$
\|Ef\|_{L^2(B_R;Hdx)}
\lessapprox
\|Ef\|_{L^{2(m+1)/(m-1)}(Y)} \left(\int_Y H\,dx\right)^{1/(m+1)}
$$
$$
\lessapprox E^{O(1)}\sigma^{-1/(m+1)}R^{-(d-m)/2(m+1)}  \| f \|_{L^2}
\left(NR^{\al/2}\right)^{1/(m+1)}\,,
$$
where $N$ is the number of $R^{1/2}$-cubes in $Y$.
Note that $N\lesssim \sigma R^{1/2}$, the above is thus further bounded by
$$
\lessapprox E^{O(1)} R^{\frac 1 2 -\frac{d-\al}{2(m+1)}},
$$
as desired.
\end{proof}

\begin{corollary}[Bilinear weighted $L^{r_m}$ estimate] \label{cor:bilLp}
Let $d\geq 2$ and $\al\in (0,d]$. Let $m$ be a dimension in the range $2\leq m \leq d$. For functions $f_1$ and $f_2$ in $L^2(B^{d-1})$, with supports separated by $\sim 1$, suppose that $f_1$ and $f_2$ are concentrated in wave packets from $\ZT_Z(E)$, where $Z=Z(P_1,\cdots,P_{d-m})$  is a transverse complete intersection with ${\rm Deg}\,P_i \leq D_Z$. Let $r_m:=\frac{2(m+1)}{m}$ and $H\in \cf_{\al,d}$. Then,
\begin{equation} \label{bilLp}
\left\| |E f_1|^{1/2}| Ef_2 |^{1/2} \right\|_{L^{r_m}(B_R;Hdx)} \lessapprox E^{O(1)} R^{-\frac{2d-2m-\al}{4(m+1)} }\| f_1 \|_{L^2}^{1/2}\| f_2 \|_{L^2}^{1/2}.
\end{equation}
\end{corollary}

\begin{proof}
Without loss of generality, assume that $\|f_1\|_{L^2}=\|f_2\|_{L^2}=1$.
We break $B_R$ into $R^{1/2}$-cubes $Q_j$. Let $\cy_{\ga_1,\ga_2}$ denote the collection of those $Q_j$'s such that 
$$
\|E f_i\|_{L^{q_m}(Q_j)} \sim \ga_i, \quad i=1,2\,.
$$
Define $Y _{\ga_1,\ga_2}:= \bigcup_{Q_j \in \cy_{\ga_1,\ga_2}} Q_j$. Note that there are only $\sim (\log R)^2$ relevant dyadic scales $(\ga_1,\ga_2)$, hence
$$
\left\| |E f_1|^{1/2}| Ef_2 |^{1/2} \right\|_{L^{r_m}(B_R;Hdx)}
\lesssim (\log R)^2
\left\| |E f_1|^{1/2}| Ef_2 |^{1/2} \right\|_{L^{r_m}(Y;Hdx)}\,,
$$
where $Y=Y_{\ga_1,\ga_2}:=\bigcup_{j=1}^{N} Q_j$ for some $(\ga_1,\ga_2)$. Therefore by H\"older's inequality and Theorem \ref{thm-bilref-m} one obtains
$$
\left\| |E f_1|^{1/2}| Ef_2 |^{1/2} \right\|_{L^{r_m}(B_R;Hdx)}
\lessapprox
\left\| |E f_1|^{1/2}| Ef_2 |^{1/2} \right\|_{L^{q_m}(Y)} \left(\int_Y H\,dx\right)^{\frac{1}{2(m+1)}}
$$
$$
\lessapprox E^{O(1)} R^{-\frac{d-m}{2(m+1)}} N^{-\frac{1}{2(m+1)}} \| f_1 \|_{L^2}^{1/2}\| f_2 \|_{L^2}^{1/2} \left(NR^{\frac \al 2}\right)^{\frac{1}{2(m+1)}}
$$
$$
= E^{O(1)} R^{-\frac{2d-2m-\al}{4(m+1)} } \,,
$$
as desired.
\end{proof}

\begin{remark} \label{rmk:falc0}
The linear weighted $L^2$ estimate in Corollary \ref{cor:linL2} in the case $m=d$ says that for all $d\geq 2$, $\al\in(0,d], f\in L^2(B^{d-1}), H\in \cf_{\al,d}$, 
\begin{equation}
\|Ef\|_{L^2(B_R;Hdx)} \lessapprox R^{\frac 1 2 -\frac{d-\al}{2(d+1)}}\|f\|_{L^2}\,.
\end{equation}
Therefore it follows from Proposition \ref{wtRestoFalc} that for a compact subset $E$ of $\ZR^d$,
\begin{equation} \label{falc0}
{\rm dim}(E)> \frac d 2 + \frac 1 4 + \frac{3}{8d+4} \Rightarrow |\Delta(E)|>0.
\end{equation}
This already improves Erdo\~gan's result \cite{Erdg05} for $d > 4$. In addition, by Remark \ref{rmk:wtRestoAD} one obtains
\begin{equation} \label{AvrDec0}
\be_d(\al) \geq \al-1+\frac{d-\al}{d+1}\,.
\end{equation}
This improves Erdo\~gan's result \cite{Erdg05} in the range $\al> \frac d 2 + \frac{1}{d-1}$ and Luc\`a-Rogers' result \cite{LR} in the full range. 
\end{remark}

The results \eqref{falc0} on Falconer's problem and \eqref{AvrDec0} on average decay rates in the above can be further improved to Theorem \ref{AvrDec-wtRes}, by combining refined Strichartz estimates and the method of polynomial partitioning developed by the second author \cite{lG, lG16}. That will be the content of the rest of the paper.

\section{Weighted extension estimates in three dimensions:\\ proof of Theorem 
\ref{AvrDec-wtRes} for $d=3$}
\label{sec:wtRes-3}
\setcounter{equation}0

In this section, we prove the case $d=3$ of Theorem \ref{AvrDec-wtRes} using polynomial partitioning and bilinear refined Strichartz estimates, building on the work of \cite{lG, DGL}.

For any $\al\in(0,3]$ and $p>3$, we will prove that
\begin{equation} \label{eq:wtRes-3'}
\|Ef\|_{L^p(B_R;Hdx)} \lessapprox R^{\gamma^0_3(\alpha)}\|f\|_{L^2}
\end{equation}
holds for all $f\in L^2(B^2)$, all $R> 1$ and all $H\in \cf_{\al,3}$, where
\[
\gamma^0_{3}(\alpha)=\begin{cases}
0, & \alpha\in (0,2],\\
\frac{\alpha}{3}-\frac{2}{3}, & \alpha\in (2,3],
\end{cases}
\]

The weighted $L^3$ estimate in Theorem \ref{AvrDec-wtRes} follows from H\"older's inequality and the estimate \eqref{eq:wtRes-3'} by taking $p\rightarrow 3^+$. Note that one can assume $R$ is sufficiently large, as otherwise the bound \eqref{eq:wtRes-3'} becomes trivial. The proof uses induction on the physical radius $R$. 

\subsection{Polynomial partitioning and cell contributions}
We pick a degree $D=R^{\delta_{\rm deg}}$, where $\delta_{\rm deg}\ll\delta\e$. By the polynomial partitioning theorem (c.f. Theorem 1.4 in \cite{lG}), there exists a non-zero polynomial
$P$ of degree at most $D$ such that $\ZR^3\backslash Z(P)$ is a union of $\sim D^{3}$
disjoint open sets $O_i$ and for each $i$ there holds
\begin{equation} 
  \|Ef\|^p_{L^p(B_R;Hdx)} \sim D^3 \|Ef\|^p_{L^p(B_R\cap O_i;Hdx)}.
\end{equation}
Moreover, the polynomial $P$ is a product of distinct non-singular polynomials.

Define the \emph{wall}
\begin{equation}
  W:=N_{R^{1/2+\delta}}Z(P)\cap B_R\,,
\end{equation}
where $\delta\ll\e$ and  $N_{R^{1/2+\delta}}Z(P)$ stands for the $R^{1/2+\delta}$-neighborhood of
the variety $Z(P)$ in $\ZR^3$. For each cell $O_i$, set
\begin{equation}
 O'_i :=\left[O_i\cap B_R\right] \backslash W
\,\,\,{\rm and}\,\,\,
 \ZT_i:=\{(\theta,\nu)\in\ZT\,:\,T_{\theta,\nu} \cap O'_i \neq \emptyset\} \,.
\end{equation}
For each function $f$, define
\begin{equation}
  f_i :=  \sum_{(\theta,\nu)\in\ZT_i} f_{\theta,\nu} \,.
\end{equation}
Then on each cell $O_i'$,
\begin{equation}
 Ef(x)\sim Ef_i(x),\quad x\in O_i' \,.
\end{equation}
By the fundamental theorem of algebra, we have a simple geometric observation: for each $(\theta,\nu)$,
$$
\#\{i : T_{\theta,\nu}\cap O_i' \neq \emptyset \} \leq D+1\,.
$$
This geometric observation and orthogonality allow us to control the $L^2$ norms of $f_i$'s:
\begin{equation}
\sum_i \|f_i\|_{L^2}^2 \lesssim D\|f\|^2_{L^2}\,.
\end{equation}

Break $\|Ef\|^p_{L^p(B_R;Hdx)}$ into
\begin{equation*}
\sum_i  \|Ef\|^p_{L^p(O_i';Hdx)}
+  \|Ef\|^p_{L^p(W;Hdx)}\,,
\end{equation*}
and call it the \emph{algebraic} case if the wall contribution $\|Ef\|^p_{L^p(W;Hdx)}$ dominates. We first consider the \emph{non-algebraic} case, where the main contribution comes from the cells $O_i'$. In the non-algebraic case, 
\begin{equation}
\|Ef\|^p_{L^p(B_R;Hdx)} \sim D^3 \|Ef\|^p_{L^p(O_i';Hdx)}
\end{equation}
still holds for $\sim D^3$ indices $i$'s, and among which by pigeonholing one can pick an $i_0$ such that
\begin{equation}
\|f_{i_0}\|_{L^2}^2 \lesssim D^{-2}\|f\|^2_{L^2}\,. 
\end{equation}
Now the non-algebraic case can be handled by induction:
$$
\|Ef\|^p_{L^p(B_R;Hdx)} \sim D^3 \|Ef\|^p_{L^p(O'_{i_0};Hdx)} \lesssim D^3 \|Ef_{i_0}\|^p_{L^p(B_R;Hdx)}
$$
$$
\lesssim D^3 \left(R^{\e+\gamma^0_3(\alpha)} \|f_{i_0}\|_{L^2}\right)^p
\lesssim D^{3-p} \left(R^{\e+\gamma^0_3(\alpha)}\|f\|_{L^2}\right)^p\,.
$$
Recall that $D=R^{\delta_{\rm deg}}$ and $R$ is assumed to be sufficiently large compared to any constant depending on $\e$, therefore $D^{3-p}\ll 1$ provided that $p>3$, thus the induction closes.

\subsection{Wall contribution}
To deal with the wall term $\|Ef\|^p_{L^p(W;Hdx)}$, we break $B_R$ into $\sim R^{3\delta}$ balls $B_j$ of radius $R^{1-\delta}$.

For any tile $(\theta,\nu) \in\ZT$,
$T_{\theta,\nu}$ is said to be \emph{tangent} to the wall $W$ in a given ball $B_j$ if
it satisfies that $T_{\theta,\nu}\cap B_j\cap W \neq \emptyset$ and
\begin{equation}
 \text{Angle}(G(\theta),T_z[Z(P)])\leq R^{-1/2+2\delta}
\end{equation}
for any non-singular point $z\in 10T_{\theta,\nu}\cap 2B_j\cap Z(P)$.
Recall that $G(\theta)\in S^{2}$ is the direction of the tube $T_{\theta,\nu}$. Here $T_z[Z(P)]$ stands for the tangent space to the variety $Z(P)$ at the point $z$, and by a non-singular point we mean a point $z$ in $Z(P)$ with $\nabla P(z) \neq 0.$ Since $P$ is a product of distinct non-singular polynomials, the non-singular points are dense in $Z(P)$.  We note that if $T_{\theta, \nu}$ is tangent to $W$ in $B_j$, then $T_{\theta, \nu} \cap B_j$ is contained in the $R^{1/2 + \delta}$-neighborhood of $Z(P) \cap 2 B_j$.  

We say that $T_{\theta,\nu}$ is \emph{transverse}  to the wall $W$ in the ball $B_j$ if it enjoys the property
that  $T_{\theta,\nu}\cap B_j\cap W \neq \emptyset$ and
\begin{equation}
 \text{Angle}(G(\theta),T_z[Z(P)])> R^{-1/2+2\delta}
\end{equation}
for some non-singular point $z\in 10T_{\theta,\nu}\cap 2B_j\cap Z(P)$.

Let $\ZT_{j, {\rm tang}}$ represent the collection of
all tiles $(\theta,\nu)\in \ZT$ such that $T_{\theta,\nu}$'s are tangent to the wall $W$ in $B_j$,
and $\ZT_{j, {\rm trans}}$ denote the collection of
all tiles $(\theta,\nu)\in \ZT$ such that $T_{\theta,\nu}$'s are transverse to the wall $W$ in $B_j$.

Define $f_{j,{\rm tang}}:=\sum_{(\theta,\nu)\in \ZT_{j,{\rm tang}}}f_{\theta,\nu}$ and $f_{j,{\rm trans}}:=\sum_{(\theta,\nu)\in \ZT_{j,{\rm trans}}}f_{\theta,\nu}$.
Then on $B_j \cap W$, $Ef(x)$ can be split into a transverse term and a tangential term:
\begin{equation}
Ef(x)\sim Ef_{j,\rm tang}(x)+Ef_{j,\rm trans}(x)\,.
\end{equation}
However, since we will need to use a bilinear structure when analyzing the tangent contribution, here we use a more refined decomposition instead: breaking $Ef(x)$ into a linear transverse term and a bilinear tangential term.

More precisely, decompose the unit ball $B^2$ into balls $\tau$ of radius $1/K$, where $K=K(\epsilon)\ll R$ is a large parameter. Decompose $f=\sum_{\tau}f_\tau$, where $\text{supp}\,f_\tau \subseteq \tau$.

Let $B_\epsilon :=\{x\in B_R\,:\,\exists\, \tau\,\, \text{s.t.}\, |Ef_\tau(x)|>K^{-\e^4}|Ef(x)|\}$.  We will show by parabolic rescaling that
the contribution from $B_\e$ is acceptable. In fact, by the definition of $B_\e$,
\begin{equation*}
  \big\|Ef (x)\big\|^p_{L^p(B_\e;Hdx)}
  \leq K^{\epsilon^4 p}\sum_\tau \big\|Ef_\tau(x)\big\|^p_{L^p(B_R;Hdx)} \,. 
\end{equation*}
By parabolic rescaling and induction on scales (Lemma \ref{ParResc}), the right hand side is bounded by
\begin{align}
  \lesssim \, & K^{\e^4p}\sum_\tau \left[K^{\frac{\al+1}{p}-1-\e-\ga^0_3} R^{\e+\gamma^0_3(\alpha)} \|f_\tau\|_{L^2}\right]^p \notag \\
  \lesssim \,& K^{(\e^4+\frac{\al+1}{p}-1-\e-\ga^0_3)p} \left[R^{\e+\gamma^0_3(\alpha)}\|f\|_{L^2}\right]^p. \notag
\end{align}
Note that $\frac{\al+1}{p}-1-\ga^0_3 \leq 0$ (this is the reason why we set $\ga^0_3 = \frac{\al-2}{3}$ for $\al>2$). By choosing $K=K(\epsilon)$ large enough so that $$K^{(\e^4+\frac{\al+1}{p}-1-\e-\ga^0_3)} \ll 1,$$
the induction closes and therefore the term involving $B_\epsilon$ plays 
an unimportant role.

For points not in $B_\epsilon$, we have the following decomposition into a transverse term and a bilinear tangential term (cf. \cite[Lemma 6.2]{DGL}):

\begin{lemma} \label{L:wall}
For each point $x\in B_j\cap W$ satisfying $\max_{\tau} |Ef_\tau(x)|\leq K^{-\epsilon^4}|Ef(x)|$,
there exists a sub-collection $I$ of the collection of all possible $1/K$-balls $\tau$, such that
\begin{equation}
|Ef(x)| \lesssim |Ef_{I,j,\rm trans}(x)| + K^{10}
{\rm Bil}(Ef_{j, \rm tang}(x)),
\end{equation}
where
$$
f_{I,j,\rm trans}:=\sum_{\tau \in I} f_{\tau, j, \rm trans},
$$
and the bilinear tangent term is given by
$$
{\rm Bil}(Ef_{j,\rm tang}(x)):=\max_{\substack{\tau_1,\tau_2\\ \text{dist}(\tau_1,\tau_2)\geq 1/K}}
|Ef_{\tau_1,j,\rm tang}(x)|^{1/2}|Ef_{\tau_2,j,\rm tang}(x)|^{1/2}.
$$
\end{lemma}

By Lemma \ref{L:wall} we bound the wall term $\|Ef\|^p_{L^p(W;Hdx)}$ by
\begin{align}
\lesssim & \sum_{j}\big\|\max_{I} \,|Ef_{I,j,\rm trans} (x)|\big\|^p_{L^p(B_j\cap W;Hdx)}\label{tran0}\\
  +&K^{10p} \sum_j\big\|{\rm Bil}(Ef_{j,\rm tang} (x))\big\|^p_{L^p(B_j\cap W;Hdx)} \,.\label{Bil}
\end{align}
We handle the transverse term by induction on physical radius and control the $L^2$ norms of $f_{j,\rm trans}$ using the following Lemma, which says that $T_{\theta,\nu}$ crosses the wall $W$ transversely in at most $R^{O(\delta_{\rm deg})}$
many balls $B_j$.

\begin{lemma} [Lemma 3.5 in \cite{lG}]
  For each tile $(\theta,\nu) \in \ZT$, the number of $R^{1-\delta}$-balls $B_j$ for which
 $(\theta,\nu)\in \ZT_{j,{\rm trans}}$ is at most $\text{Poly}(D)=R^{O(\delta_{\rm deg})}$.
\end{lemma}

The above geometric lemma and orthogonality imply the bound:
$$
\sum_j\|f_{j,\rm trans}\|_{L^2}^2 \leq R^{O(\delta_{\rm deg})}\|f\|_{L^2}^2.
$$
We now estimate the linear transverse term (\ref{tran0}). The term (\ref{tran0}) is dominated by
\begin{equation}\label{maxI}
 \sum_{j}\sum_{I\subseteq \mathcal T}\big\|Ef_{I,j,\rm trans} (x)\big\|^p_{L^p(B_j\cap W;Hdx)}\, ,
\end{equation}
where $\mathcal T$ is the collection of all possible $1/K$-balls in $B^2$, and the sum is taken over all
subsets of $\mathcal T$. Since there are at most $2^{K^2}$ $I$'s,  we apply \eqref{eq:wtRes-3'} with radius $R^{1-\delta}$ to obtain
$$
 (\ref{maxI}) \leq \sum_{j}  2^{K^2}  \left[C_\epsilon R^{(1-\delta)(\epsilon+\ga^0_3)}\|f_{j,\rm trans}\|_{L^2}\right]^p
$$
$$
 \lesssim 2^{K^2} R^{O(\delta_{\rm deg})-\delta\epsilon p-\delta\ga^0_3 p}  \left[C_\epsilon R^{\epsilon+\ga^0_3}\|f\|_{L^2}\right]^p\,.
$$
Since $\delta_{\rm deg} \ll \delta \e$, it follows that
$2^{K^2} R^{O(\delta_{\rm deg})-\delta\epsilon p-\delta\ga^0_3 p} \ll 1$, thus the induction on the transverse term closes. \\

It remains to estimate the bilinear tangent term (\ref{Bil}). The proof uses the bilinear refined Strichartz. By Corollary \ref{cor:bilLp} in the case $m=2$ and $d=3$, we have the following: 

Let $\al\in(0,3]$. 
For functions $f_1$ and $f_2$ in $L^2(B^2)$, with supports separated by $\sim 1$, suppose that $f_1$ and $f_2$ are concentrated in wave packets from $\ZT_Z(E)$, where $Z=Z(P)$ and $P$ is a product of distinct non-singular polynomials.
Then for any $H\in \cf_{\al,3}$,
\begin{equation} \label{bilL3}
\left\| |E f_1|^{1/2}| Ef_2 |^{1/2} \right\|_{L^3(B_R;Hdx)} \lessapprox E^{O(1)}R^{(\al-2)/12}\| f_1 \|_{L^2}^{1/2}\| f_2 \|_{L^2}^{1/2}.
\end{equation}

Now we estimate the bilinear tangent term \eqref{Bil}.
$$
\eqref{Bil} \leq K^{10} \sum_j \sum_{\substack{\tau_1,\tau_2\\ \text{dist}(\tau_1,\tau_2)\geq 1/K}}
\left\| |E f_{\tau_1,j,\rm tang}|^{1/2}| Ef_{\tau_2,j,\rm tang} |^{1/2} \right\|^p_{L^p(B_j;Hdx)}\,.
$$
To finish the proof of the estimate \eqref{eq:wtRes-3'}, it suffices to show
\begin{equation} \label{bilL3'}
\left\| |E f_{\tau_1,j,\rm tang}|^{1/2}| Ef_{\tau_2,j,\rm tang} |^{1/2} \right\|_{L^3(B_j;Hdx)} \lessapprox 
R^{\ga^0_3}\| f_{\tau_1} \|_{L^2}^{1/2}\| f_{\tau_2} \|_{L^2}^{1/2}\,,
\end{equation}
for each pair $(\tau_1,\tau_2)$ with $\rm{dist}(\tau_1,\tau_2)\geq 1/K$.
We will do so by applying \eqref{bilL3} to $f_{\tau_i,j, \rm tang}$ on each ball $B_j$.  Expand $f_{\tau_i,j, \rm tang}$ into wave packets at the scale $\rho = R^{1 - \delta}$ on the ball $B_j$. By definition of $f_{\tau_i,j, \rm tang}$, each wave packet lies in the $\sim R^{1/2 + \delta}$-neighborhood of $Z$ and the angles between the wave packets and the tangent space of $Z$ are bounded by $R^{-1/2 + 2 \delta}$.  For a detailed description of the wave packet decomposition of $f_{\tau_i,j, tang}$ on a smaller ball, see \cite[Section 7]{lG16}.  Define $E$ so that $\rho^{1/2} E = R^{1/2 + \delta}$.  Since $\rho = R^{1 - \delta}$, there holds $E = R^{(3/2) \delta}$, thus $E \rho^{-1/2} = R^{-1/2 + 2 \delta}$.  Each new wave packet lies in the $E \rho^{1/2}$-neighborhood of $Z$, and the angles between the wave packets and the tangent space of $Z$ are bounded by $E \rho^{-1/2}$.  Therefore, the new wave packets are concentrated in $\ZT_Z(E)$, which enables one to apply \eqref{bilL3}.  Now since $E^{O(1)} = R^{O(\delta)}$ and $(\al-2)/12 \leq \ga^0_3$, the bound from \eqref{bilL3} implies \eqref{bilL3'}. The proof is complete.

\section{Weighted extension estimates in higher dimensions:\\ proof of Theorem \ref{wtRes-d}} \label{sec:wtRes-d}
\setcounter{equation}0

In this section, we prove Theorem \ref{wtRes-d} using polynomial partitioning. Roughly speaking, we will iterate the argument in Section \ref{sec:wtRes-3} in each dimension. Because of the complexity of the iteration scheme and some technical issues, we present the argument using the notion of \emph{narrow} and \emph{broad} part of $Ef$. The broad part, which is the main body of the proof, is estimated by Theorem \ref{wt2br} below, and the narrow part is handled by Lemma \ref{wtnr} based on parabolic rescaling. 

To start with, fixing a large constant $K$, we decompose $B^{d-1}$ into balls $\tau$ of radius $K^{-1}$ and $B_R$ into balls $B_{K^2}$ of radius $K^2$. One naturally has $f=\sum_\tau f_\tau:=\sum_\tau f\chi_\tau$, and $G(\tau)$ denotes the set of directions of wave packets of $f_\tau$. We use $\text{Angle}(G(\tau),V)$ to denote the smallest angle between $v\in G(\tau)$ and $v'\in V$, any vector space in $\mathbb{R}^d$. We are now ready to define the following \emph{broad} norm of $Ef$. Fix $H\in \mathcal{F}_{\alpha,d}$,
$$
\|Ef\|^p_{BL^p_A(B_R;Hdx)}:=\sum_{B_{K^2}\subset B_R} \mu_{Ef}(B_{K^2}),
$$where
$$
\mu_{Ef}(B_{K^2}):=\min_{V_1,\ldots,V_A :1-\text{subspace of }\mathbb{R}^d}\left(\max_{\tau:\,\text{Angle}(G(\tau),V_a)>K^{-1}\text{ for all }a}\int_{B_{K^2}}|Ef_\tau|^p H\,dx\right).
$$Here $A$ is a large constant to be determined later. When the value of $A$ is not important we usually write $\|\cdot\|_{BL^p_A}$ as $\|\cdot\|_{BL^p}$ for short. 
We can extend $\mu_{Ef}$ to be a measure on $B_R$, making it a constant multiple of the Lebesgue measure on each ball $B_{K^2}$. 

The $k$-broad norm was first invented by the second author in \cite{lG16} where it is used as a weaker substitute for the $k$-linear norm but still strong enough to imply linear restriction estimate. The broad norm we are using here is the same as the one in \cite{lG16} in the case $k=2$ except that the measure $Hdx$ is used instead of the Lebesgue measure. The constant $A$ is introduced to ensure that the broad norm satisfies some versions of triangle inequality and H\"older's inequality. We refer the reader to \cite{lG16} for more detailed discussion on properties of broad norm.

The main chunk of the proof of Theorem \ref{wtRes-d} is the following estimate.
\begin{theorem}\label{wt2br}
Let $d\geq 4$, $\alpha\in [\frac{d}{2},\frac{d+1}{2}]$ and $p_d:=\frac{2d}{d-1}$. For all $\epsilon>0$, there is a large constant $A$ so that the following holds for any value of $K$, $R>1$, $H\in\mathcal{F}_{\alpha,d}$: 
$$
\|Ef\|_{BL^{p_d}(B_R;Hdx)}\lesssim_{K,\epsilon} R^{\epsilon+\gamma_d}\|f\|_{L^2(B^{d-1})},
$$where $\gamma_d:=\frac{\alpha}{2d^2}-\frac{1}{4d}$.
\end{theorem}

Note that Theorem \ref{wt2br} will be further generalized to Theorem \ref{AvrDec2br} in Section \ref{sec:wtRes-d-al}. To see that Theorem \ref{wt2br} implies the desired Theorem \ref{wtRes-d}, it suffices to apply Lemma \ref{wtnr} below with $p=p_d$, and note that it is straightforward to check $$\ga_d\geq \frac{1-d}{2}+\frac{\alpha+1}{p_d}.$$

\begin{lemma}\label{wtnr}
Let $d\geq 3$, $p\geq 2$ and $\alpha\in (0,d]$. Assume that for all $\epsilon>0$, there exists large constant $A=A(\e)$ such that 
\begin{equation} \label{wtnr-br}
\|Ef\|_{BL^{p}_A(B_R;Hdx)}\lesssim_{K,\epsilon} R^{\epsilon+T_d}\|f\|_{L^2(B^{d-1})}
\end{equation}
holds for all $K,R>1$,$H\in\mathcal{F}_{\alpha,d}$, and that 
\begin{equation} \label{wtnr-resc}
T_d\geq \frac{1-d}{2}+\frac{\alpha+1}{p}.
\end{equation}
Then, for all $\epsilon>0$, $R>1$, $H\in\mathcal{F}_{\alpha,d}$, there holds
\begin{equation} \label{wtnr-lin}
\|Ef\|_{L^{p}(B_R;Hdx)} \leq C_\epsilon R^{\epsilon+T_d}\|f\|_{L^2(B^{d-1})}.
\end{equation}
\end{lemma}

\begin{proof}

We write $\|Ef\|_{L^p(B_R;Hdx)}$ as
$$
 \left( \sum_{B_{K^2}\subset B_R} \big\|Ef\big\|_{L^p(B_{K^2};Hdx)}^p \right)^{1/p}\,,
$$
and for each $B_{K^2}$, take $1$-subspaces $V'_1,\cdots,V'_A$ of $\ZR^d$ depending on $B_{K^2}$ and $f$ to be the minimizers obeying
\begin{equation} \label{min'}
 \max_{\tau\notin V_a'\text{ for all }a}\int_{B_{K^2}}|Ef_\tau|^p H\,dx
=\min_{V_1,\ldots,V_A:1-\text{subspace  }}\max_{\tau \notin V_a \text{ for all }a}\int_{B_{K^2}}|Ef_\tau|^p H\,dx \,, 
\end{equation}
where $\tau \notin V_a$ means that $\text{Angle}(G(\tau),V_a)>K^{-1}$.
Then on each $B_{K^2}$ by applying Minkowski inequality to function 
$$Ef=\sum_{\tau \notin V'_a \text{ for all } a} Ef_\tau + 
\sum_{\tau \in V'_a \text{ for some }a}Ef_\tau,$$
we bound $\|Ef\|_{L^p(B_R;Hdx)}$ by 
$$
\left( \sum_{B_{K^2}\subset B_R} \left\|\sum_{\tau \notin V'_a \text{ for all } a} Ef_\tau \right\|_{L^p(B_{K^2};Hdx)}^p \right)^{1/p}
$$
$$
+
\left( \sum_{B_{K^2}\subset B_R} \left\|\sum_{\tau \in V'_a \text{ for some } a} Ef_\tau \right\|_{L^p(B_{K^2};Hdx)}^p \right)^{1/p}\,.
$$
By the choice as in \eqref{min'} and assumption \eqref{wtnr-br}, the first term is bounded by
$$
K^{O(1)} \big\| Ef \big\|_{BL^p_A (B_R; Hdx)} \lesssim_{K,\epsilon} R^{\epsilon+T_d}\|f\|_{L^2(B^{d-1})} \,.
$$
Note that there are only $O(A)$ many $\tau$'s that are ``in''
$V_1', \cdots, V_{A}'$. We choose $K=K(\e)$ large enough so that $A=A(\e)\leq K^\delta$. Henceforth, the second term is controlled by
$$
\leq K^{O(\delta)} \big\| \max_{\tau} | Ef_\tau | \big\|_{L^p(B_R;Hdx)}
\leq K^{O(\delta)}\big(\sum_{\tau} \big\|Ef_\tau\big\|^p_{L^p(B_R;Hdx)}\big)^{1/p}
\,.
$$
Note that to prove the desired estimate \eqref{wtnr-lin}, one can induct on radius $R$. Therefore by applying Lemma \ref{ParResc} which is based on parabolic rescaling and induction on physical radius, the narrow part above is further estimated by
$$
\lesssim C_\e K^{O(\delta)}K^{\frac{\al+1}{p}-\frac{d-1}{2}-\e-T_d}R^{\e+T_d} \big(\sum_\tau \|f_\tau\|_{L^2}^p\big)^{1/p}\,.
$$
Due to orthogonality and the fact $p\geq 2$, we have $\big(\sum_\tau \|f_\tau\|_{L^2}^p\big)^{1/p} \lesssim \|f\|_{L^2}$. Moreover by the assumption \eqref{wtnr-resc}, $K^{O(\delta)+\frac{\al+1}{p}-\frac{d-1}{2}-\e-T_d}\ll 1$. Therefore, the narrow part can be estimated as desired by induction and the proof is complete.
\end{proof}

It remains to prove Theorem \ref{wt2br}.
Similarly as the three-dimensional case treated in the previous section, we apply polynomial partitioning (but iteratively in different dimensions). To make use of induction on dimensions, we generalize Theorem \ref{wt2br} to the following main inductive proposition: 

\begin{proposition}\label{prop:wtRes-d}
Given $d\geq 4$, $\alpha\in [\frac{d}{2},\frac{d+1}{2}]$. For all $\epsilon>0$, there exist a large constant $\bar{A}>1$ and small constants $0<\delta\ll \delta_{d-1}\ll\ldots\ll \delta_1\ll\epsilon$ so that the following holds. Let $m$ be a dimension in the range $2\leq m\leq d$, and $p_m:=\frac{2m}{m-1}$. Suppose that $Z=Z(P_1,\ldots, P_{d-m})$ is a transverse complete intersection with ${\rm Deg}P_i\leq D_Z$, and that $f\in L^2(B^{d-1})$ is concentrated in wave packets from $\ZT_Z(R^{\delta_m})$. Then for any $1\leq A\leq\bar{A}$, $R\geq 1$ and $H\in\mathcal{F}_{\alpha,d}$,
\begin{equation}\label{wtRes-d-m}
\|Ef\|_{BL^{p_m}_A (B_R;Hdx)}\leq C(K,\epsilon,m,D_Z) R^{m\e} R^{\delta(\log\bar{A}-\log A)} R^{\ga_m}\|f\|_{L^2},
\end{equation}where
\[
\gamma_m:=\begin{cases} -\frac{d}{4m}+\frac{1}{4}, & 2\leq m\leq d-1,\\
\frac{\alpha}{2d^2}-\frac{1}{4d}, & m=d.
\end{cases}
\]
\end{proposition}
In the proposition, $f$ being concentrated in wave packets from $\ZT_Z(R^{\delta_m})$ is defined as in subsection \ref{wpd}. It is easy to see that the case $m=d$, $Z=\mathbb{R}^d$, $A=\bar{A}$ in the proposition above is precisely the desired result of Theorem \ref{wt2br}. Proposition \ref{prop:wtRes-d} will be proven by induction (on dimension $m$, radius $R$, and on $A$) with the assistance of the linear refined Strichartz in each step (more precisely, the linear weighted $L^2$ estimate in Corollary \ref{cor:linL2} for each dimension $m$). The rest of this section is devoted to the proof of Proposition \ref{prop:wtRes-d}.

%The constraint $\alpha\in [\frac{d}{2},\frac{d+1}{2}]$ comes in when comparing the transverse case to interpolation. Cellular case always works since we are using $p_m$. 

The base case $m=2$ (for all $R$ and $A$) follows immediately from the unweighted estimate (Proposition 8.1 of \cite{lG16}) and Remark \ref{unweighted}. If $R$ is small, then choosing the implicit constant large enough will finish the proof. If $A=1$, then by choosing $\bar{A}$ large enough, the desired estimate follows from the trivial $L^1\to L^\infty$ estimate of $E$. Now fix $m\leq d$ and assume that the desired estimates hold true if one decreases $m$, $R$, or $A$.

We say we are in algebraic case if there is a transverse complete intersection $Y^{m-1}\subset Z^m$ of dimension $m-1$, defined using polynomials of degree $\leq D(\e,D_Z)$ (a function to be determined later), such that
$$
\mu_{Ef}(N_{R^{1/2+\delta_m}}(Y)\cap B_R)\gtrsim \mu_{Ef}(B_R).
$$
Otherwise we say that we are in the non-algebraic (or \emph{cellular}) case.

\subsection{The non-algebraic case}
In the non-algebraic case, we use polynomial partitioning and induction on radius $R$. Since the argument is exactly the same as in Subsection 8.1 of \cite{lG16}, here we just give a brief description. 

First by pigeonholing we can locate a significant piece of $N_{R^{1/2+\delta_m}}(Z)\cap B_R$ where at each point the angle between the tangent space of $Z$ and a fixed $m$-plane $V$ is within $1/100$. Then perform the regular polynomial partitioning in $V$ and pull the polynomial on $V$ back via the orthogonal projection $\pi: \ZR^d \rightarrow V$. We end up with a polynomial $P$ on $\ZR^d$ of degree $\leq D=D(\e,D_Z)$, for which $\ZR^d\backslash Z(P)$ is a union of $\sim D^m$ open sets $O_i$ and the following properties hold. Define $W:= N_{R^{1/2+\de}}Z(P)$, $O'_i:=O_i\backslash W$ and $f_i = \sum_{(\theta,\nu)\in \ZT_i}f_{\theta,\nu}$, where
$$
\ZT_i:=\big\{(\theta,\nu):T_{\theta,\nu}\cap O'_i \neq \emptyset\big\}.
$$
Since we are in the non-algebraic case, for $\sim D^m$ cells $O'_i$, 
$$
\|Ef\|^{p}_{BL^{p}_A (B_R;Hdx)} \lesssim D^m \|Ef\|^{p}_{BL^{p}_A (O'_i;Hdx)} \lesssim D^m \|Ef_i\|^{p}_{BL^{p}_A (B_R;Hdx)}\,.
$$
In addition, by orthogonality and the geometric observation that each $(\theta,\nu)$ belongs to $\lesssim D$ collections $\ZT_i$, we have
$$
\sum_i \|f_i\|_{L^2}^2 \lesssim D \|f\|_{L^2}^2\,.
$$
Therefore, by the same argument as in three dimensions, the induction for the non-algebraic case closes provided that $p>p_m=\frac{2m}{m-1}$. Then, applying H\"older's inequality and letting $p\to p_m^+$ will justify the same estimate for the endpoint $p=p_m$.

\subsection{The algebraic case}
In the algebraic case, there exists a transverse complete intersection $Y$ of dimension $m-1$, defined using polynomials of degree $\leq D(\epsilon,D_Z)$ such that 
\[
\mu_{Ef}(N_{R^{1/2+\delta_m}}(Y)\cap B_R)\gtrsim \mu_{Ef}(B_R).
\]

In this case, we first subdivide $B_R$ into smaller balls $B_j$ of radius $\rho$, chosen such that $\rho^{1/2+\delta_{m-1}}=R^{1/2+\delta_m}$. One has
\[
\|Ef\|^{p_m}_{BL^{p_m}_A(B_R;Hdx)}\lesssim\sum_{j}\|Ef_j\|^{p_m}_{BL^{p_m}_A(B_j;Hdx)}+{\rm RapDec}(R)\|f\|_{L^2}^{p_m},
\]where
\[
f_j:=\sum_{(\theta,\nu)\in\ZT_j} f_{\theta,\nu},\quad \ZT_j:=\{(\theta,\nu):\,T_{\theta,\nu}\cap N_{R^{1/2+\delta_m}}(Y)\cap B_j\neq \emptyset\}.
\]
Similarly as in Section \ref{sec:wtRes-3}, we further subdivide $\ZT_j$ into tubes that are tangent to $Y$ and tubes that are transverse to $Y$. We say that $T_{\theta,\nu} \in \ZT_j$ is \emph{tangent} to $Y$ in $B_j$ if
\begin{equation}
T_{\theta,\nu}\cap 2B_j\subset N_{R^{1/2+\delta_m}}(Y) \cap 2B_j =
N_{\rho^{1/2+\delta_{m-1}}}(Y) \cap 2B_j
\end{equation}
and for any non-singular point $y\in Y \cap 2B_j\cap N_{10R^{1/2+\delta_m}}T_{\theta,\nu}$,
\begin{equation}
 \text{Angle}(G(\theta),T_y Y)\leq \rho^{-1/2+\delta_{m-1}}\,.
\end{equation}

We denote the tangent and transverse wave packets by
\[
\ZT_{j,{\rm tang}}:=\{(\theta,\nu)\in\ZT_j:\, T_{\theta,\nu} \text{ is tangent to } Y \text{ in } B_j\},\quad \ZT_{j,{\rm trans}}:=\ZT_j\setminus \ZT_{j,{\rm tang}},
\]and let
\begin{equation}\label{eqn:trans}
f_{j,{\rm tang}}=\sum_{(\theta,\nu)\in\ZT_{j,{\rm tang}}} f_{\theta,\nu},\quad f_{j, {\rm trans}}=\sum_{(\theta,\nu)\in\ZT_{j,{\rm trans}}} f_{\theta,\nu},
\end{equation}then
\begin{align*}
\sum_{j}\|Ef_j\|^{p_m}_{BL^{p_m}_A(B_j;Hdx)}\lesssim &\sum_j\|Ef_{j,{\rm tang}}\|^{p_m}_{BL^{p_m}_{A/2}(B_j;Hdx)}\\+&\sum_j\|Ef_{j,{\rm trans}}\|^{p_m}_{BL^{p_m}_{A/2}(B_j;Hdx)}.
\end{align*}
We will control the contribution from the tangent wave packets by induction of the dimension $m$, and the one from the transverse wave packets by induction on the radius $R$.

\subsection{The tangent sub-case}
In this subsection, we control the tangent term
$$
\sum_j\|Ef_{j,{\rm tang}}\|^{p_m}_{BL^{p_m}_{A/2}(B_j;Hdx)}
$$
by induction on dimension $m$. 
In order to apply the induction hypotheses to $Ef_{j,{\rm tang}}$ on $B_j$, one needs to first redo the wave packet decomposition at the scale $\rho$. 
By definition of $\ZT_{j,\rm tang}$, it is easy to check, as in the $3$-dimensional case in the previous section, that such a wave packet $T_{\tilde\theta,\tilde\nu}$ of dimensions $\rho^{1/2+\delta}\times \cdots\times \rho^{1/2+\delta}\times \rho$ is $\rho^{-1/2+\delta_{m-1}}$-tangent to $Y$ in $B_j$, in other words, $f_{j,{\rm tang}}$ satisfies the hypotheses of Proposition \ref{prop:wtRes-d} at scale $\rho$ in dimension $m-1$. Therefore, by induction on the dimension one has
\begin{equation}\label{eqn:INT1}
\begin{split}
&\|Ef_{j,{\rm tang}}\|_{BL^{p_{m-1}}_{A/2}(B_j;Hdx)}\\
\leq &C(K,\epsilon,m-1,D(\epsilon,D_Z)) \rho^{(m-1)\e}\rho^{\delta(\log\bar A-\log (A/2))} \rho^{\gamma_{m-1}}\|f_{j,{\rm tang}}\|_{L^2}.
\end{split}
\end{equation}

On the other hand, it follows immediately from the definition of the broad norm and Corollary \ref{cor:linL2} that
\begin{equation}\label{eqn:INT2}
\|Ef_{j,\rm tang}\|^2_{BL^2(B_j;Hdx)}\leq \sum_{\tau}\|Ef_{\tau,j,\rm tang}\|^2_{L^2(B_j;Hdx)}
\end{equation}
$$
\leq C_\e  \rho^{O(\delta_{m-1})}\rho^{\epsilon+1 -\frac{d-\al}m}\sum_\tau\|f_{\tau,j,\rm tang}\|^2_{L^2}
\leq C_\e  \rho^{O(\delta_{m-1})}\rho^{\epsilon+1 -\frac{d-\al}m} \|f_{j,\rm tang}\|^2_{L^2}.
$$
Observing that $2<p_m=\frac{2m}{m-1}<\frac{2(m-1)}{m-2}=p_{m-1}$, one can interpolate estimates (\ref{eqn:INT1}), (\ref{eqn:INT2}) above to obtain
\begin{align*}
&\|Ef_{j,{\rm tang}}\|_{BL^{p_m}_{A/2}(B_j;Hdx)}\\ \leq &C_{K,\epsilon,m-1,D(\e,D_Z)}\rho^{O(\delta_{m-1})+(m-1)\epsilon+\delta(\log\bar{A}-\log(A/2))+\gamma_{m,m-1}}\|f_{j,{\rm tang}}\|_{L^2},
\end{align*}
where
\[
\gamma_{m,m-1}=(\frac{1}{2}-\frac{d-\alpha}{2m})\cdot\frac{1}{m}+\gamma_{m-1}\cdot(1-\frac{1}{m}).
\]Note that the number of balls $B_j$ is $\lesssim R^{O(\delta_{m-1})}$, hence one can sum over the balls to obtain
\begin{equation}\label{eqn:tang-d}
\begin{split}
&\big(\sum_j\|Ef_{j,{\rm tang}}\|^{p_m}_{BL^{p_m}_{A/2}(B_j;Hdx)}\big)^{1/p_m}\\
&\leq C_{K,\epsilon,m-1,D(\e,D_Z)}
R^{O(\delta_{m-1})+(m-1)\epsilon+\delta(\log\bar{A}-\log(A/2))}\rho^{\gamma_{m,m-1}}\|f\|_{L^2}\\
&\leq C_{K,\epsilon,m,D(\e,D_Z)} R^{O(\delta_{m-1})+(m-1)\e+(\log 2)\delta}R^{\delta(\log\bar{A}-\log A)+\gamma_{m,m-1}}\|f\|_{L^2}.
\end{split}
\end{equation}In the last inequality above, even though $\gamma_{m,m-1}$ can be negative, one still has $\rho^{\gamma_{m,m-1}}\leq R^{O(\delta_{m-1})}R^{\gamma_{m,m-1}}$. Since $\delta, \delta_{m-1}\ll\epsilon$, $$R^{O(\delta_{m-1})+(m-1)\epsilon+(\log 2)\delta}\ll R^{m\e}\,,$$ 
hence the inductive argument for the tangent term is done as long as 
\begin{equation}
\ga_m\geq (\frac{1}{2}-\frac{d-\alpha}{2m})\cdot\frac{1}{m}+\gamma_{m-1}\cdot(1-\frac{1}{m})\,.
\end{equation}

\subsection{The transverse sub-case}
In this subsection we deal with the transverse term
$$
\sum_j\|Ef_{j,{\rm trans}}\|^{p_m}_{BL^{p_m}_{A/2}(B_j;Hdx)}
$$
by induction on the radius $R$. The argument is exactly the same as in the Subsection 8.4 of \cite{lG16}, hence we omit the details and only briefly recall several essential steps. 

As in the tangent sub-case, in order to apply induction on radius, we need to redo wave packet decomposition for $f_{j,\rm trans}$ at scale $\rho$. Since the old relevant wave packets are in $\ZT_Z(R^{\delta_m})$, for a new relevant wave packet $T_{\tilde\theta,\tilde\nu}$ of dimensions $\rho^{1/2+\delta}\times\cdots\times\rho^{1/2+\delta}\times\rho$, the angle between $G(\tilde \theta)$ and the tangent spaces of $Z$ near their intersection is $\lesssim R^{-1/2+\delta_m}+\rho^{-1/2} \lesssim \rho^{-1/2+\delta_m}$. We decompose $N_{R^{1/2+\delta_m}}(Z)\cap B_j$ into translates of $N_{\rho^{1/2+\delta_m}}(Z) \cap B_j$, say $N_{\rho^{1/2+\delta_m}} (Z+b) \cap B_j$, $|b|\leq  R^{1/2+\delta_m}$. Define $f_{j,{\rm trans},b}$ using the new wave packets which intersect $N_{\rho^{1/2+\delta_m}} (Z+b) \cap B_j$. Because of the angle condition, $f_{j,{\rm trans}, b}$ is concentrated in new wave packets that are $\rho^{-1/2+\delta_m}$-tangent to $Z+b$ inside $B_j$. We can choose a set of translations $\{b\}$ such that
\begin{equation}\label{trans-b}
\|Ef_{j,{\rm trans}}\|^{p_m}_{BL^{p_m}_{A/2}(B_j;Hdx)} \lesssim (\log R)
\sum_b \|Ef_{j,{\rm trans},b}\|^{p_m}_{BL^{p_m}_{A/2}(B_j;Hdx)}\,.
\end{equation}
By orthogonality and Lemma 5.7 in \cite{lG16} which controls the transverse intersections between a tube and an algebraic variety, one has
\begin{equation} \label{trans-ortho}
\sum_{j,b}\|f_{j,{\rm trans},b}\|_{L^2}^2 \lesssim \|f\|_{L^2}^2\,.
\end{equation}
Moreover, there holds the equi-distribution estimate (c.f. Section 7 of \cite{lG16})
\begin{equation} \label{trans-equi}
\max_b \|f_{j,{\rm trans},b}\|_{L^2}^2
\leq R^{O(\delta_m)} \left(\frac{R^{1/2}}{\rho^{1/2}}\right)^{-(d-m)} \|f_{j,\rm trans}\|_{L^2}^2\,.
\end{equation}
By inductive hypothesis we can apply \eqref{wtRes-d-m} to $\|Ef_{j,{\rm trans},b}\|_{BL^{p_m}_{A/2}(B_j;Hdx)}$ to obtain
\begin{equation*}
\begin{split}
&\sum_j\|Ef_{j,{\rm trans}}\|^{p_m}_{BL^{p_m}_{A/2}(B_j;Hdx)} 
\lesssim (\log R)\sum_{j,b} \|Ef_{j,{\rm trans},b}\|^{p_m}_{BL^{p_m}_{A/2}(B_j;Hdx)}\\
\lesssim &(\log R) \sum_{j,b} \big[\rho^{m\e} \rho^{\delta(\log\bar{A}-\log (A/2))} \rho^{\ga_m}\|f_{j,{\rm trans},b}\|_{L^2}\big]^{p_m}\,.
\end{split}
\end{equation*}
It follows from \eqref{trans-ortho} and \eqref{trans-equi} that
$$
\sum_{j,b}\|f_{j,{\rm trans},b}\|_{L^2}^{p_m} \leq R^{O(\delta_m)}\left(\frac{R^{1/2}}{\rho^{1/2}}\right)^{-(d-m)(\frac{p_m}{2}-1)} \|f\|_{L^2}^{p_m}\,,
$$
therefore,
\begin{equation*}
\begin{split}
&\sum_j\|Ef_{j,{\rm trans}}\|^{p_m}_{BL^{p_m}_{A/2}(B_j;Hdx)}\\
\lesssim & R^{O(\delta_m)} \big[\rho^{m\e} R^{\delta(\log\bar{A}-\log A)} \rho^{\ga_m}\big]^{p_m} \left(\frac{R^{1/2}}{\rho^{1/2}}\right)^{-(d-m)(\frac{p_m}{2}-1)} \|f\|_2^{p_m}\,.
\end{split}
\end{equation*}
Choosing $\delta_m \ll \e\delta_{m-1}$, one has $$R^{O(\delta_m)}\left(\frac{R}{\rho}\right)^{-m\e}=R^{O(\delta_m)}R^{-O(\e\delta_{m-1})} \ll 1.$$ Henceforth the induction closes as long as 
$$
\frac 12 (d-m)\big(\frac{p_m}{2}-1\big) + p_m \ga_m \geq 0\,,
$$
that is,
\begin{equation} \label{eq:trans}
\ga_m\geq -\frac{d}{4m}+\frac{1}{4}\,.
\end{equation}

\subsection{Summary}
Because of the inductive argument for the non-algebraic case, the exponent $p_m=\frac{2m}{m-1}$ is the smallest possible one can work with. Starting with 
$$\ga_2=-\frac{d}{8}+\frac 14,$$
the algebraic case gives the constraint
$$
\ga_m \geq \max\left\{-\frac{d}{4m}+\frac{1}{4}\,,\,(\frac{1}{2}-\frac{d-\alpha}{2m})\cdot\frac{1}{m}+\gamma_{m-1}\cdot(1-\frac{1}{m}) \right\}
$$
It is straightforward to check that in the range $\alpha\in[\frac{d}{2},\frac{d+1}{2}]$, one can take
\[
\ga_m =\begin{cases}
-\frac{d}{4m}+\frac{1}{4}, & 2\leq m\leq d-1,\\
\frac{\alpha}{2d^2}-\frac{1}{4d}, & m=d.
\end{cases}
\]
This completes the proof of Proposition \ref{prop:wtRes-d}.

\section{Generalized weighted extension estimates in higher dimensions:\\ proof of Theorem \ref{AvrDec-wtRes} for $d\geq 4$} \label{sec:wtRes-d-al}
\setcounter{equation}0

In this section, we prove Theorem \ref{AvrDec-wtRes} for $d\geq 4$, which generalizes Theorem \ref{wtRes-d} to the full range of $\al$. Same as in Section \ref{sec:wtRes-d}, Theorem \ref{AvrDec-wtRes} is a result of the following broad extension estimate and Lemma \ref{wtnr}.

\begin{theorem}\label{AvrDec2br}
Let $d\geq 4$, $\alpha\in (0,d]$ and $p_d=\frac{2d}{d-1}$. For all $\epsilon>0$, there is a large constant $A$ so that the following holds for any value of $K,R>1$ and any $H\in\mathcal{F}_{\alpha,d}$:
\[
\|Ef\|_{BL^{p_d}(B_R;Hdx)}\lesssim_{K,\epsilon}R^{\epsilon+\gamma_d(\alpha)}\|f\|_{L^2(B^{d-1})},
\]where
\[
\gamma_d(\alpha):=\begin{cases}
\frac{(1+2S_4^d)\alpha}{4d}-\frac{1}{2d}-\frac{S^d_4}{2},& \alpha\in (d-1, d],\\
\frac{S^d_4 \alpha}{2d}+\frac{1}{4}-\frac{3}{4d}-\frac{S^d_4}{2},& \alpha\in (d-2, d-1],\\
\frac{S^d_\ell \alpha}{2d}+\frac{1}{4}-\frac{\ell-1}{4d}-\frac{S^d_\ell}{2}, &\alpha\in (d-\frac{\ell}{2}, d-\frac{\ell}{2}+\frac{1}{2}],\,\forall 5\leq\ell\leq d,\\
0,& \alpha\in (0,\frac{d}{2}],
\end{cases}
\]and $S_\ell^d:=\sum_{i=\ell}^d\frac{1}{i}$ if $\ell\leq d$, $0$ otherwise.
\end{theorem}

To prove Theorem \ref{AvrDec-wtRes}, recall that according to Lemma \ref{wtnr}, an estimate for the broad part implies the same estimate for the regular $L^p$ norm as long as condition $$\gamma_d(\alpha)\geq\frac{1-d}{2}+\frac{\alpha+1}{p_d}$$
is satisfied. It is straightforward to check that this is indeed the case when $$\alpha\leq\#_d:=\frac{2d(d-2-S^d_4)}{2d-3-2S^d_4}.$$ When $\#_d<\alpha\leq d$, in order for the narrow part to be controlled, the best bound one can get from the broad estimate above is
\[
\|Ef\|_{L^{p_d}(B_R;Hdx)}\leq C(\epsilon)R^{\epsilon+\frac{1-d}{2}+\frac{\alpha+1}{p_d}}\|f\|_{L^2},
\]which is exactly the desired estimate for $\alpha\in (\#_d,d]$ in Theorem \ref{AvrDec-wtRes}. 

We also point out that when $\alpha\in (\frac{d}{2},\frac{d+1}{2}]$, the estimate in Theorem \ref{AvrDec2br} coincides with Theorem \ref{wt2br}.

It remain to prove Theorem \ref{AvrDec2br}. The proof follows rom the same strategy as Theorem \ref{wt2br}, where the main tools are polynomial partitioning and induction on scales and dimensions. To make all inductions work, we formulate the following main inductive proposition in a more general setting:

\begin{proposition}\label{prop:AvrDec1}
Given $d\geq 4$, $\alpha\in (0,d]$. For all $\epsilon>0$, there exist a large constant $\bar{A}>1$ and small constants $0<\delta\ll \delta_{d-1}\ll\ldots\ll \delta_1\ll\epsilon$ so that the following holds. Let $m$ be a dimension in the range $3\leq m\leq d$, and $p_m:=\frac{2m}{m-1}$. Suppose that $Z=Z(P_1,\ldots, P_{d-m})$ is a transverse complete intersection with ${\rm Deg}P_i\leq D_Z$, and that $f\in L^2(B^{d-1})$ is concentrated in wave packets from $\ZT_Z(R^{\delta_m})$. Then for any $1\leq A\leq\bar{A}$, $R\geq 1$ and $H\in\mathcal{F}_{\alpha,d}$,
\begin{equation} \label{eq:AvrDec1}
\|Ef\|_{BL_A^{p_m}(B_R;Hdx)}\leq C(K,\epsilon,m,D_Z) R^{m\e}R^{\delta(\log\bar{A}-\log A)} R^{\ga_m}\|f\|_{L^2},
\end{equation}where
\[
\gamma_m(\alpha):=\begin{cases} \frac{\alpha}{12}-\frac{d}{6}+\frac{1}{3}, & m=3,\\
\frac{(1+2S_4^m)\alpha}{4m}+\frac{m-1}{2m}-\frac{(1+S_4^m)d}{2m}, & 4\leq m\leq d,
\end{cases}\quad \text{ if }\alpha\in (d-1,d];
\]
\[
\gamma_m(\alpha):=\begin{cases} -\frac{d}{12}+\frac{1}{4}, & m=3,\\
\frac{S_4^m\alpha}{2m}+\frac{2m-3}{4m}-\frac{(1+2S_4^m)d}{4m}, & 4\leq m\leq d,
\end{cases}\quad \text{ if }\alpha\in (d-2,d-1];
\]
\[
\begin{split}
\gamma_m(\alpha):=\begin{cases} -\frac{d}{4m}+\frac{1}{4}, & 3\leq m\leq \ell-1,\\
\frac{S_\ell^m \alpha}{2 m}+\frac{2m-\ell+1}{4m}-\frac{(1+2S_\ell^m)d}{4 m}, & \ell\leq m\leq d,
\end{cases}&\\
\text{ if }\alpha\in (d-\frac{\ell}{2},d-\frac{\ell}{2}+\frac{1}{2}],\,\forall 5\leq\ell\leq d;&
\end{split}
\]
\[
\gamma_m(\alpha):=-\frac{d}{4m}+\frac{1}{4},\quad 3\leq m\leq d,\quad \text{ if }\alpha\in (0,\frac{d}{2}].
\]
\end{proposition}

Theorem \ref{AvrDec2br} follows from Proposition \ref{prop:AvrDec1} by taking $m=d$, $Z=\mathbb{R}^d$ and $A=\bar{A}$. And Proposition \ref{prop:AvrDec1} coincides with Proposition \ref{prop:wtRes-d} when $\alpha\in (\frac{d}{2},\frac{d+1}{2}]$.

The proof of Proposition \ref{prop:AvrDec1} proceeds very similarly as Proposition \ref{prop:wtRes-d}. To begin with, assume $m=3$. To validate the inductive argument for the non-algebraic case, the exponent $p_3=3$ is the smallest possible one we can work with. The transverse case gives a constraint \eqref{eq:trans}:
$$
\ga_3(\al)\geq -\frac{d}{12}+\frac{1}{4}.
$$
As for the tangent sub-case, recall that by interpolating with an $L^2$ estimate which is based on linear refined Strichartz, we have an estimate with essential exponent 
$$
\ga_{3,2}=\frac{\al}{18}-\frac{5d}{36}+\frac{1}{3}\,.
$$
On the other hand, by the bilinear weighted $L^3$ estimate in Corollary \ref{cor:bilLp} (it follows from a randomization argument that $k$-linear estimate is stronger than $k$-broad estimate, cf. \cite{GHI}), we have another estimate for the tangent term with essential exponent
$$
\ga'_{3,2}= \frac{\al}{12}-\frac{d}{6}+\frac{1}{3}\,.
$$
In summary, we have the estimate \eqref{eq:AvrDec1} when $m=3$ with
$$
\ga_3(\al)=\max\left\{-\frac{d}{12}+\frac{1}{4}, \min\big\{\frac{\al}{18}-\frac{5d}{36}+\frac{1}{3}, \frac{\al}{12}-\frac{d}{6}+\frac{1}{3}
\big\}\right\}\,.
$$
Note that 
$$\frac{\al}{18}-\frac{5d}{36}+\frac{1}{3}\geq  \frac{\al}{12}-\frac{d}{6}+\frac{1}{3}$$for all $\al\leq d$, meaning that the bilinear refined Strichartz works better than the linear refined Strichartz in this case. And 
$$
-\frac{d}{12}+\frac{1}{4}\geq \frac{\al}{12}-\frac{d}{6}+\frac{1}{3}
$$
for $\al\leq d-1$. This completes the proof for the base case $m=3$. 

Now, fix $4\leq m\leq d$ and assume that the desired estimates hold true if one decreases $m$, $R$, or $A$.
From the same argument as in the previous section, we have the desired estimate \eqref{eq:AvrDec1} with
$$
\ga_m(\al)=\max\left\{-\frac{d}{4m}+\frac{1}{4}\,,\,(\frac{1}{2}-\frac{d-\alpha}{2m})\cdot\frac{1}{m}+\gamma_{m-1}(\al)\cdot(1-\frac{1}{m})\right\}\,,
$$
where the second exponent in the above is a consequence of interpolation with the $L^2$ estimate in Corollary \ref{cor:linL2} implied by the linear refined Strichartz estimate, $\forall\alpha\in (0,d]$. Note that even though for certain $m$, bilinear refined Strichartz would provide a better bound (i.e. a smaller exponent) for the tangent contribution, it would not translate into a better $\gamma_m(\alpha)$ due to the constraint from the transverse contribution (i.e. the first exponent in the above). 

It remains to check that one can indeed take $\ga_m(\al)$ as stated in Proposition \ref{prop:AvrDec1}, which follows from straightforward computation and is left to the reader.

\subsection{Comparison of tools}\label{subsec:compare}
%Compare tools: interpolation, bilinear refined Strichartz, H\"older's inequality, linear refined Strichartz.
There are various tools that have been used in the argument above and in Section \ref{sec:wtRes-3}, \ref{sec:wtRes-d}, such as linear and bilinear refined Strichartz estimates, which we would like to discuss a bit more and compare in this subsection.

First, as pointed out in Remark \ref{rmk:falc0}, applying linear refined Strichartz estimate directly, one can immediately obtain some result on Falconer's problem for $d > 4$, which is already better than the previously best known bounds but is not as good as our Theorem \ref{Falc}. This is because the strategy of combining refined Strichartz and polynomial partitioning becomes more and more effective as $\alpha$ decreases from $d$ to $\frac{d}{2}$.

Second, in the proof of Proposition \ref{prop:wtRes-d} and \ref{prop:AvrDec1}, we have studied the tangent sub-case using interpolation between the induction hypothesis from one dimension lower and the weighted $L^2$ estimate in Corollary \ref{cor:linL2} which is based on linear refined Strichartz. Alternatively, one may instead apply directly H\"older's inequality or bilinear weighted estimate in Corollary \ref{cor:bilLp} which is based on bilinear refined Strichartz.

More precisely, Corollary \ref{cor:bilLp} can be applied for each $m$, similarly as in the proof of the base case $m=3$ of Proposition \ref{prop:AvrDec1}, to obtain an estimate for the tangent term. Or, H\"older's inequality implies that
\begin{equation}
\begin{split}
&\|Ef_{j,{\rm tang}}\|_{BL^{p_m}(B_j;Hdx)}\\
\lesssim &\|Ef_{j,{\rm tang}}\|_{BL^{p_{m-1}}(B_j;Hdx)}\left(\int_{N_{R^{1/2+\delta_{m-1}}}(Y)\cap B_j} H\,dx \right)^{\frac{1}{p_m}-\frac{1}{p_{m-1}}},
\end{split}
\end{equation}
which, combined with the fact that $H\in \cf_{\al,d}$ and the induction hypothesis on $(m-1)$-dimensional varieties, produces another estimate for the contribution from tangent wave packets. 

Both estimates already yield improvement of previously best known results for Falconer's problem and the Fourier decay rates of fractal measures, but are weaker than our Theorem \ref{Falc} and Theorem \ref{AvrDec}. Roughly speaking, the method involving H\"older's inequality produces the weakest result among all options, the method via interpolation is the best when $m$ is larger than $d/2$, otherwise the bilinear refined Strichartz approach rules. However, as already mentioned in the proof of Proposition \ref{prop:AvrDec1}, it turns out that it is unnecessary to apply the stronger bilinear refined Strichartz even if $m$ is small, which is because in this case there is too much constraint from the transverse sub-case.

\end{document}